\documentclass[12pt]{amsart}
\usepackage{bbm}
\usepackage{xcolor}
\usepackage[normalem]{ulem}
\usepackage{yfonts}

\usepackage[bookmarks,colorlinks]{hyperref}
\usepackage{amssymb}
\theoremstyle{plain}
\newtheorem{theorem}{Theorem}[section]

\newtheorem{lemma}[theorem]{Lemma}
\newtheorem{proposition}[theorem]{Proposition}

\theoremstyle{definition}
\newtheorem{definition}{Definition}[section]

\theoremstyle{remark}
\newtheorem*{rem*}{Remark}

\makeatletter
\@addtoreset{equation}{section}
\makeatother

\date{\today}
\author[Wang]{Dinghuai Wang}
\address{School of Mathematics and Statistics, Anhui Normal University, Wuhu, 241002, China}
\email{Wangdh1990@126.com}
\author[Hu]{Xi Hu}
\address{School of Mathematical Sciences, Beijing Normal University,
 Laboratory of Mathematics and Complex Systems, Ministry of Education, Beijing, 100875, China.}
 \email{huximath1994@pku.edu.cn}
\author[Qi]{Shuai Qi}
\address{ School of Mathematical Sciences, Beijing Normal University, Laboratory of Mathematics and Complex Systems, Ministry of Education, Beijing, 100875, China}
\email{qshuai@bnu.edu.cn}

\thanks{This work was supported by National Natural Science Foundation of China(Nos. 12101010,11771023) and Natural Science Foundation of China of Anhui Province (No. 2108085QA19).}



\begin{document}
	\title[Endpoint theory for the compactness of commutators]{Endpoint theory for the compactness of commutators}

	\begin{abstract} In this paper, we establish a Minkowski-type inequality for weak Lebesgue space, which allows us to obtain a characterization of relative compactness in these spaces. Furthermore, we are the first to investigate the compactness results of commutators at the endpoint.
The paper provides a comprehensive study of the compactness properties of commutators of Calder\'{o}n-Zygmund operators in Hardy and $L^{1}(\mathbb{R}^n)$ type spaces.
Additionally, we provide factorization theorems for Hardy spaces in terms of singular integral operators in the $L^1(\mathbb{R}^n)$ space.
	\end{abstract}\
	\maketitle
	
\section{Introduction and Statement of Main Results}

The endpoint theory for the compactness of commutators has been missing until now, and this paper aims to fill that gap. The main objective of this study is to investigate the compactness properties of commutators of singular integral operators with suitable symbols. In particular, we derive a Minkowski-type inequality for weak Lebesgue spaces and establish the characterization of relative compactness in these spaces.

We briefly summarize some classical and recent works in the literature, which lead to the results presented here (precise definitions are given in the next section). Let's define $b$ as a locally integrable function on $\mathbb{R}^n$ and $T$ as a Calder\'{o}n-Zygmund singular integral operator. Now, let's look at the commutator operator $[b,T]$, which is defined for smooth functions $f$ as
$$[b,T]f=bT(f)-T(bf).$$
It is a well-established fact that commutators of singular integrals with a multiplication by a measurable function $b$ are bounded operators on $L^p(\mathbb{R}^n), 1<p<\infty$, provided $b$ is a BMO function \cite{CRW1976}. Furthermore, if the commutator of all Riesz transforms
$$R_{j}f(x)={\rm p.v.}\int_{\mathbb{R}^n}\frac{x_{j}-y_{j}}{|x-y|^{n+1}}f(y)dy, 1\leq j\leq n,$$ are bounded for some $p, 1 < p <\infty$, then the function $b$ must necessarily be in BMO. In \cite{J1978}, Janson advanced this result by demonstrating that it is sufficient to assume the boundedness of the commutator $[b, R_{j}]$ for some $j\in 1,2, \cdots, m$. Commutators of more general singular integral operators were considered by Uchiyama in \cite{U1, U2}, Li in \cite{LSY} and Wang in \cite{W-SM}. These findings have had significant applications in other areas of operator theory and partial differential equations. For instance, the investigation of div-curl lemmas \cite{CLMS1993, LPPW2012} and additional interpretations in operator theory \cite{NPTV2002,N1957} have stemmed from this work.
This theory has since been expanded and generalized in several directions. The concept of the smoothing effect of commutators of linear operators is a well-established and highly beneficial principle, where smoothing refers to the enhancement of boundedness to a more robust condition of compactness. Uchiyama \cite{U2}, in 1978, honed the boundedness results related to commutators to compactness. This refinement was achieved by mandating the commutator with the symbol to reside in ${\rm CMO}(\mathbb{R}^n)$, which is the closure in ${\rm BMO}(\mathbb{R}^n)$ of the space of $C^{\infty}$ functions with compact support. It is demonstrated that $[b,T]$ holds compactness on $L^{p}(\mathbb{R}^n)$ for any $p\in (1,\infty)$ if and only if $b\in {\rm CMO}(\mathbb{R}^n)$. In addition to this, Uchiyama posited an equivalent characterization of ${\rm CMO}(\mathbb{R}^n)$ in terms of the mean value oscillation of functions, which serves a pivotal role in the compactness characterization of $[b, T]$ on $L^{p}(\mathbb{R}^n)$.

For the case $p=1$, P\'{e}rez addressed the boundedness of commutators of singular integrals with ${\rm BMO}(\mathbb{R}^n)$ symbols in \cite{P1995}. Specifically, he showed that the commutator of a singular integral with a ${\rm BMO}(\mathbb{R}^n)$ symbol is not bounded from $L^1(\mathbb{R}^n)$ onto $L^{1,\infty}(\mathbb{R}^n)$. However, they do satisfy the modular inequality of the $L\log L(\mathbb{R}^n)$ type. P\'{e}rez also introduced a subspace of $H^1(\mathbb{R}^n)$  for which $[b, T]$ is a bounded operator, indicating that these commutators are more singular operators than Calder\'{o}n-Zygmund operators. It was also proven that $[b, T]$ is not of $(H^1(\mathbb{R}^n),L^1(\mathbb{R}^n))$ type unless $b$ is constant. Instead, Chen and Hu in \cite{CH2001} proved the weak $(H^1(\mathbb{R}^n),L^1(\mathbb{R}^n))$ boundedness.

However, the compactness of commutators in the $L^1(\mathbb{R}^n)$ type and Hardy spaces is still missing. The main goal of this work is to show the compactness of commutators of singular integrals at the endpoint. Simple examples (cf. Section 6) are provided to show that the commutator $[b, T]$ fails to be of type $(L\log L(\mathbb{R}^n), L^{1,\infty}(\mathbb{R}^n))$ \big(or $(H^{1}(\mathbb{R}^n), L^{1,\infty}(\mathbb{R}^n))\big)$ compact operators when $b$ is in ${\rm CMO}(\mathbb{R}^n)$. The study of relatively compact sets in weak Lebesgue spaces is necessary for this purpose, and the theorem obtained provides an important improvement of the known results. We obtain the following theorem:

\begin{theorem}\label{WLP}
Let $0<p<\infty$. A subset $\mathcal{F}$ of $L^{p,\infty}(\mathbb{R}^n)$ is relatively compact if and only if the following three conditions hold:
\begin{enumerate}
  \item[(i)] norm boundedness uniformly
  \begin{equation}\label{WLP1}
  \sup_{f\in \mathcal{F}}\|f\|_{L^{p,\infty}(\mathbb{R}^n)}<\infty;
  \end{equation}
  \item[(ii)] translation continuty uniformly
  \begin{equation}\label{WLP2}
\lim_{a\rightarrow 0}\sup_{y\in B(O,a)}\|f(\cdot+y)-f(\cdot)\|_{L^{p,\infty}(\mathbb{R}^n)}=0 \ \text{uniformly in} f\in \mathcal{F};
  \end{equation}
  \item[(iii)] control uniformly away from the origin
    \begin{equation}\label{WLP3}
\lim_{A \rightarrow \infty}\|f\chi_{E_{A}}\|_{L^{p,\infty}(\mathbb{R}^n)}=0 \  \text{uniformly in} f\in \mathcal{F},
  \end{equation}
  where $E_{A}=\{x\in \mathbb{R}^n:|x|>A\}$.
\end{enumerate}
\end{theorem}

Also, P\'{e}rez in \cite{P1995} considered endpoint estimates related to Hardy-types spaces and introduced a subspace of $H^{1}(\mathbb{R}^n)$ for which $[b,T]$ is a bounded operator. Now, we give the corresponding compactness result as follows.

\begin{theorem}\label{thmHardyb}
Assuming that $b\in {\rm CMO}(\mathbb{R}^n)$ and $T$ is a Calder\'{o}n-Zygmund operator, then the commutator $[b,T]$ is a compact operator mapping from $H^{1}_{b}(\mathbb{R}^n)$ to $L^{1}(\mathbb{R}^n)$.
\end{theorem}

Alternatively, Janson \cite{J1978} in 1978 employed a Fourier expansion technique to establish that for $0<\alpha<1$, $b\in Lip_{\alpha}(\mathbb{R}^n)$ if and only if the commutator $[b,T]$ with a smooth kernel is bounded from $L^{p}(\mathbb{R}^n)$ to $L^{q}(\mathbb{R}^n)$ for $1<p<q<\infty$ with $1/q=1/p-\alpha/n$. In a recent study, Nogayamaand and Sawano in \cite{NS2017}, Guo, He, Wu, and Yang \cite{GHWY} explored a ${\rm CMO}$ type space ${\rm CMO}_{\alpha}(\mathbb{R}^n)$ and the compactness of commutators of singular or fractional integral operators, where $1<p<q<\infty$ and $1/q=1/p-\alpha/n$. We extend their work to establish the corresponding compactness results for the commutator in $L^{1}(\mathbb{R}^n)$ and Hardy spaces.

\begin{theorem}\label{thmHardy}
Assuming that $0<\alpha<1$, $b\in {\rm CMO}_{\alpha}(\mathbb{R}^n)$, and $T$ is a Calder\'{o}n-Zygmund operator, then the commutator $[b,T]$ is a compact operator mapping from $H^{1}(\mathbb{R}^n)$ to $L^{\frac{n}{n-\alpha}}(\mathbb{R}^n)$.
\end{theorem}

\begin{theorem}\label{thmL1}
Suppose that $0<\alpha<1$ and $T$ is a Calder\'on-Zygmund operator that is homogeneous. Then the commutator
$$[b,T]:L^{1}(\mathbb{R}^n)\rightarrow L^{\frac{n}{n-\alpha},\infty}(\mathbb{R}^n)$$
is compact if and only if $b\in {\rm CMO}_{\alpha}(\mathbb{R}^n).$
\end{theorem}

This paper is organized as follows. This paper discusses the compactness of commutators of Calder\'{o}n-Zygmund operators in various function spaces. The main results are the compactness criteria for commutators in weak Lebesgue spaces, Hardy spaces, and $L^1(\mathbb{R}^n)$.
In Section 2, the paper introduces the necessary definitions and notation for the framework being used. This section can be skipped by readers familiar with the subject.
Section 3 presents the proof of the compactness criteria for commutators in weak Lebesgue spaces.
Section 4 focuses on the compactness results for commutators in Hardy spaces.
In Section 5, the paper establishes the characterization of compactness of commutators in $L^1(\mathbb{R}^n)$ space. This section also includes some basic lemmas on Hardy factorization.
The Appendix contains counterexamples and an Minkowski-type inequality for weak Lebesgue space.

\section{Preliminaries}
Let $|E|$ denote the Lebesgue measure of a measurable set $E\subset \mathbb{R}^n$. Throughout this paper, the letter $C$ denotes constants which are independent of main variables and may change from one occurrence to another. By $A\lesssim B$ we mean that $A\lesssim CB$ with a positive constant $C$ independent of the appropriate quantities. If $A\lesssim B$ and $B\lesssim A$, we write $A\approx B$ and say that $A$ and $B$ are equivalent.

\subsection{Calder\'on-Zygmund operators}
Let $K(x,y),x,y\in \mathbb{R}^n,$ be a locally integrable function, defined away from the diagonal $\{x=y\}$. Then, we say that $K$ is a Calder\'on-Zygmund kernel if it satisfies the following size and smoothness conditions:
$$|K(x,y)|\leq \frac{C_{0}}{|x-y|^{n}},$$
for some constant $C_{0}>0$ and for all $(x,y)\in \mathbb{R}^{2n}$ away from the diagonal;
$$|K(x,y)-K(x',y)|+|K(y,x)-K(y,x')|\leq \frac{C_{0}|x-x'|^{\gamma}}{|x-y|^{n+\gamma}}$$
for some $\gamma>0$ and $|x-x'|\leq \frac{1}{2}|x-y|$. Suppose $T$ is a operator mapping from $\mathcal{S}(\mathbb{R}^n)$ to $\mathcal{S}'(\mathbb{R}^n)$, where we denote by $\mathcal{S}(\mathbb{R}^n)$ the spaces of all Schwartz function on $\mathbb{R}^n$ and by $\mathcal{S}'(\mathbb{R}^n)$ its dual space. We further assume $T$ is associated with the Calder\'on-Zygmund kernel,
\begin{equation}\label{C-Z}
T(f)(x)=\int_{\mathbb{R}^n}K(x,y)f(y)dy,
\end{equation}
whenever $f\in \mathbb{R}^n$ and $x\notin supp (f)$. If $T$ is bounded from $L^{2}(\mathbb{R}^n)$ to  $L^{2}(\mathbb{R}^n)$, then $T$ is called a Calder\'on-Zygmund operator.

We state our main results as follows. In addition, we say that $T$ is homogeneous if the kernel $K$ satisfies
\begin{equation}
K(x,y)\geq \frac{C}{M^{n}r^{n}} \qquad {\text or} \qquad K(x,y)\leq -\frac{C}{M^{n}r^{n}}
\end{equation}
for all $x\in B_{1}, y\in B_{2}$, where $B_{1}=(x_{1},r), B_{2}=(x_{2},r)$ are the disjoint balls satisfying the condition that $|x_{1}-x_{2}|\approx Mr$ with $r>0$ and $M>10$.

\subsection{The BMO type spaces}
Let $0\leq \alpha<1$. A locally integrable function $f$ is said to belong to Campanato space ${\rm BMO}_{\alpha}(\mathbb{R}^n)$ if there exists a constant
$C > 0$ such that for any cube $Q\subset \mathbb{R}^n$,
$$\frac{1}{|Q|}\int_{Q}|f(x)-f_{Q}|dx\leq C|Q|^{\alpha/n},$$
where $f_{Q}=\frac{1}{|Q|}\int_{Q}f(x)dx$ and the minimal constant $C$ is defined by $\|f\|_{\rm BMO_{\alpha}(\mathbb{R}^n)}$.

The Campanato spaces extend the notion of functions of bounded mean oscillation and allow a full characterization $Lip_{\alpha}(\mathbb{R}^n)$.
The Lipschitz (H\"{o}lder) spaces and Campanato spaces are related by the following equivalences:
$$\|f\|_{Lip_{\alpha}(\mathbb{R}^n)}:=\sup_{x,h\in \mathbb{R}^n,h\neq 0}\frac{|f(x+h)-f(x)|}{|h|^{\alpha}}\approx \|f\|_{\mathcal{C}_{\alpha,q}},\quad 0<\alpha<1.$$
The equivalence can be found in \cite{DS1984}, or \cite{WZTmn} for the general case.

In particular, ${\rm BMO}_{0}(\mathbb{R}^n)={\rm BMO}(\mathbb{R}^n)$, the spaces of bounded mean oscillation. The crucial property of ${\rm BMO}$ functions is the John-Nirenberg inequality \cite{JN1961},
$$|\{x\in Q: |f(x)-f_{Q}|>\lambda\}|\leq c_{1}|Q|e^{-\frac{c_{2}\lambda}{\|f\|_{\rm BMO(\mathbb{R}^n)}}},$$
where $c_{1}$ and $c_{2}$ depend only on the dimension. A well-known immediate corollary of the John-Nirenberg inequality is as follows:
$$\|f\|_{\rm BMO(\mathbb{R}^n)}\approx \sup_{Q}\frac{1}{|Q|}\Big(\int_{Q}|f(x)-f_{Q}|^{p}dx\Big)^{1/p},$$
for each $1<p<\infty$. In fact, the equivalence also holds for $0<p<1$. See, for example, the work of Str\"{o}mberg \cite{Str1979}(or \cite{HT2019} and \cite{WZTams} for the general case).

\subsection{The Hardy spaces}
The theory of Hardy spaces is vast and complicated, it has been systematically developed and plays an important role in harmonic analysis and PDEs, see \cite{C1974, FS1972, L1978}. A bounded tempered distribution $f$ is in the Hardy space $H^{\rho}(\mathbb{R}^n)$ if the Poisson maximal function
$$M(f;P)=\sup_{t>0}|(P_{t}*f)(x)|$$
lies in $L^{\rho}(\mathbb{R}^n)$.

We first recall the atomic decomposition of Hardy spaces.
Let $0<\rho\leq 1\leq q\leq \infty, \rho\neq q$ and the integer $l= [n(\frac{1}{\rho}-1)]$ ($[x]$ indicates the integer part of $x$).
Then $l=0$ if $\frac{n}{n+1}<\rho\leq 1$.

\begin{definition}\label{defH}
A function $a\in L^{q}(\mathbb{R}^n)$ is called a $(\rho,q,l)$ atom for $H^\rho(\mathbb{R}^n)$ if there exists a cube $Q$ such that
\begin{itemize}
  \item [\it (i)]  $a$ is supported in $Q$;
  \item [\it (ii)] $\|a\|_{L^{q}(\mathbb{R}^n)}\leq |Q|^{\frac{1}{q}-\frac{1}{\rho}}$;
  \item [\it (iii)] $\int_{\mathbb{R}^n} a(x)x^{\nu}dx=0$ for all multi-indices $\alpha$ with $0\leq |\nu|\leq l$.
  \end{itemize}
Here, $(\it i)$ means that an atom must be a function with compact support,
$(\it ii)$ is the size condition of atoms, and $(\it iii)$ is called the cancellation moment
condition. The atomic Hardy space $H^{\rho,q,l}(\mathbb{R}^n)$ is defined by
\begin{equation*}
\begin{aligned}
H^{\rho,q,l}(\mathbb{R}^n)=\Big\{f\in \mathcal{S}':f=^{S'}\sum_{k}\lambda_{k}a_{k}(x), &\text{each} ~~a_{k} ~~\text{is a} ~~(\rho,q,l)-\text{atom},\\
&\text{and}~~ \sum_{k}|\lambda_{k}|^{\rho}<\infty\Big\}.
\end{aligned}
\end{equation*}
Setting $H^{\rho,q,l}(\mathbb{R}^n)$ norm of $f$ by
$$\|f\|_{H^{\rho,q,l}(\mathbb{R}^n)}=\inf \big(\sum_{k}|\lambda_k|^{\rho}\big)^{1/\rho},$$
where the infimum is taken over all decompositions of $f=\sum_{k}\lambda_{k}a_{k}$ above.
\end{definition}
Note that $H^{\rho,q,l}(\mathbb{R}^n)=H^{\rho}(\mathbb{R}^n)$ was proved by Coifman \cite{C1974} for $n=1$ and Latter \cite{L1978} for $n>1$. This indicates that each element in $H^{\rho}(\mathbb{R}^n)$ can be decomposed into a sum of atoms in a certain way.
Note that in {\cite{P.B}}, the authors show that the dual of $H^\rho(\mathbb{R} ^{n} )$ is $Lip_\alpha(\mathbb{R} ^{n} )$; a key fact that will be used later on in this paper.

\begin{definition}\label{defHb}
A function $a$ is a $b-$atom if there is a cube $Q$ for which
\begin{enumerate}
  \item[{\it(i)}] supp $(a) \subset Q$;
  \item[{\it(ii)}] $\|a\|_{L^{\infty}}\leq \frac{1}{|Q|}$;
  \item[{\it(iii)}] $\int_{Q}a(y)dy=0$;
    \item[{\it(iv)}] $\int_{Q}a(y)b(y)dy=0$.
\end{enumerate}

The space $H^{1}_{b}(\mathbb{R}^n)$ consists of the subspace of $L^{1}(\mathbb{R}^n)$ of function $f$ which can be written as $f=\sum_{j}\lambda_{j}a_{j}$ where $a_{j}$ are $b-$atoms and $\lambda_{j}$ are complex numbers with $\sum_{j}|\lambda_{j}|<\infty$.
\end{definition}

\subsection{Relatively compact sets in quasi-Banach function spaces}

We first recall some basic definitions of function spaces. In this paper, we only consider the class of Lebesgue measurable functions, denoted by $L^{m}$, where $m$ means the Lebesgue measure on $\mathbb{R}^n$.
\begin{definition}
A (quasi-)normed space $(E,\|\cdot\|_{E})$ with $E\subset L(m)$ is called a (quasi-)Banach function space (Q-BFS) if it satisfies the following conditions:
\begin{enumerate}
  \item[(B0)] if $\|f\|_{E}=0\Longleftrightarrow f=0\ a.e.$;
  \item[(B1)] if $f\in E$, then $\||f|\|_{E}=\|f\|_{E}$;
  \item[(B2)] if $0\leq g\leq f$, then $\|g\|_{E}\leq \|f\|_{E}$;
  \item[(B3)] if $0\leq f_{n}\uparrow f$, then $\|f_{n}\|_{E}\uparrow \|f\|_{E}$;
  \item[(B4)] if $A\subset \mathbb{R}^n$ is bounded, then $\chi_{A}\in E$.
\end{enumerate}
\end{definition}

Moreover, we recall the following two definitions.
\begin{definition}
(Absolutely continuous quasi-norm). Let $E$ be a $Q$-BFS. A function $f$ in $E$ is said to have absolutely continuous quasi-norm in $E$ if $\|f\chi_{A_{n}}\|_{E}\rightarrow 0$ as $A_{n}\rightarrow 0$. The set of all functions in $E$ with absolutely continuous quasi-norm is denoted by $E_{a}.$ if $E=E_{a}$, then the space $E$ is said to have absolutely continuous quasi-norm.
\end{definition}

We point out that the dominated convergence theorem holds in $Q$-BFS with absolutely continuous quasi-norm; see \cite[Proposition 3.9]{CGB}.
\begin{definition}
(Uniformly absolutely continuous quasi-norm (UAC)). Let $K$ be a $Q$-BFS and let $K\subset E_{a}$. Then $K$ is said to have uniformly absolutely continuous quasi-norm $(K\subset UAC(E))$ if for every sequence $\{A_{k}\}_{k=1}^{\infty}$ with $A_{k}\rightarrow \emptyset$, $\|f\chi_{A_{k}}\|_{E}\rightarrow 0$ holds uniformly for all $f\in K$.
\end{definition}

\subsection{Morrey space and its predual}
Morrey spaces describe local regularity more precisely than $L^{q}(\mathbb{R}^n)$ spaces and can be seen as a complement of $L^{q}(\mathbb{R}^n)$.
\begin{definition}
Let $0\leq \alpha<n$ and $1<q <\infty $, The Morrey space $L^{q,\alpha}(\mathbb{R} ^{n})$ is defined by
$$L^{q,\alpha}(\mathbb{R} ^{n})=\{f\in L^q_{loc}(\mathbb{R} ^{n}):\|f\|_{L^{q,\alpha}(\mathbb{R}^n)}<\infty \},$$
with
$$\|f\|_{L^{q,\alpha}(\mathbb{R}^n)}:= \underset{x\in \mathbb{R} ^{n},r>0}{\sup} \left(r^{-\alpha}\int_{B(x,r)}|f(y)|^qdy \right )^{\frac{1}{q}}<\infty.$$
where the supremum is taken over all balls $B(x,r)$ in $ \mathbb{R} ^{n}$.
\end{definition}
Following Blasco, Ruiz and Vega \cite{B.O}, we define the function called a {\it block}.
\begin{definition}\label{sec2-def2.1}
Let $\alpha \in [0,n)$, $1<q<\infty$, and $\frac{1}{q}+\frac{1}{{q}' }  =1$. A function $b(x)$ is called a $(q,\alpha)$-block, if there there exists a ball $B(x_0,r)$ such that
$$supp(b)\subset B(x_0,r),\qquad\|b\|_{L^q}\le r^{-\frac{\alpha}{{q}'}}.$$

We further recall the definition of $\mathcal{B} ^{q,\alpha}(\mathbb{R}^n)$ via $(q,\alpha)$-blocks from \cite{B.O}. It was shown in \cite{B.O} that $\mathcal{B}^{q,\alpha}(\mathbb{R}^n)$ is a Banach space, and the dual space of $\mathcal{B}^{q,\alpha}$ is $L^{q',\alpha}(\mathbb{R}^n)$.
\end{definition}
\begin{definition}\label{sec2-def2.2}
Let $\alpha \in [0,n)$, $1<q<\infty$. The space $\mathcal{B} ^{q,\alpha}(\mathbb{R}^n)$ is defined by setting
$$\mathcal{B}^{q,\alpha}(\mathbb{R}^n)=\bigg\{g\in L^1_c(\mathbb{R}^n ): g=\sum^{\infty}_{j=1}m_j b_j, \sum^{\infty}_{j=1}|m_j|<\infty\bigg\},$$
where $\{b_j\}_{j\ge1}$ are $(q,\alpha)$-block. Furthermore, for every $g \in \mathcal{B} ^{q,\alpha}(\mathbb{R}^n)$, let
$$\|g\|_{\mathcal{B} ^{q,\alpha}(\mathbb{R}^n)}=\text{inf}\bigg\{\sum^{\infty}_{j=1}|m_j|\bigg\},$$
where the infimum is taken over all possible decompositions of $g$ as above.
\end{definition}

\section{Characterization of relative compactness in the weak Lebesgue spaces}

Let $L^{0}(m)$ denote the class of functions in $L(m)$ that are finite almost everywhere, with the topology of convergence in measure on sets of finite measure. We recall that $Q$-BFS is continuously embedded in $L^{0}(m)$.
\begin{lemma}\label{meas}
(Lemma 3.3 in \cite{CGB}). Let $E$ be a $Q$-BFS. Then $E$ is continuously embedded in $L^{0}(m)$. In particular, if $f_{k}$ tends to $f$ in $E$, then $f_{k}$ tends to $f$ in measure on sets of finite measure and hence some sequence convergence pointwise to $f$ almost everywhere.
\end{lemma}

\begin{lemma}\label{RC}
(Theorem 3.17 in \cite{CGB}). Let $E$ be a Q-BFS and let $K\subset E_{a}$. Then
$K$ is relatively compact in $E$ if and only if it is locally relatively compact in measure
and $K\subset UAC(E)$.
\end{lemma}

Now, we give the proof of characterization that a subset
in $L^{p,\infty}(\mathbb{R}^n)$ is a strongly pre-compact set, which is in itself interesting.

\vspace{0.3cm}

\textbf{Proof of Theorem \ref{WLP}. }
We will initially present the proof for the sufficiency. For the case $1<p<\infty$, we define the mean value of $f$ in $\mathcal{F}$ by
$$S_{a}(f)(x)=\frac{1}{|B(0,a)|}\int_{|y|\leq a}f(x+y)dy,$$
where $a>0$.

By the Minkowski-type inequality for $L^{p,\infty}(\mathbb{R}^n)$ with $1<p<\infty$ (see Proposition \ref{M-WLP} in Appendix), we have
\begin{equation}\label{W1}
\begin{aligned}
\|S_{a}f-f\|_{L^{p,\infty}(\mathbb{R}^n)}&\leq \bigg\|\frac{1}{|B(0,a)|}\Big|\int_{|y|\leq a}f(\cdot+y)-f(\cdot)dy\Big|\bigg\|_{L^{p,\infty}(\mathbb{R}^n)}\\
&\lesssim\frac{1}{|B(0,a)|}\int_{|y|\leq a}\|f(\cdot+y)-f(\cdot)\|_{L^{p,\infty}(\mathbb{R}^n)}dy\\
&\lesssim\sup_{|y|\leq a}\|f(\cdot+y)-f(\cdot)\|_{L^{p,\infty}(\mathbb{R}^n)}.
\end{aligned}
\end{equation}
It follows from \eqref{WLP1}, \eqref{WLP2} and \eqref{W1} that
\begin{equation}\label{a0}
\lim_{a\rightarrow 0}\|S_{a}f-f\|_{L^{p,\infty}(\mathbb{R}^n)}=0, \ \text{uniformly\ in} \ f\in \mathcal{F}
\end{equation}
and the set
$\{S_{a}f:f\in \mathcal{F}\}\subset L^{p,\infty}(\mathbb{R}^n)$
satisfies $$\sup_{f\in \mathcal{F}}\|S_{a}f\|_{L^{p,\infty}(\mathbb{R}^n)}\lesssim 1.$$ By \eqref{WLP3}, for any $0<\epsilon<1$, there exist $N>0$ and $A$ such that
\begin{equation}\label{EA1}
1<\epsilon^{-N}/4<A^{n/p}<\epsilon^{-N/2},
\end{equation}
and for every $f\in \mathcal{F}$,
\begin{equation}\label{EA2}
\|f_{E_{A}}\|_{L^{p,\infty}(\mathbb{R}^n)}<\epsilon/8.
\end{equation}

Now we prove that for each fixed $a$, the set $\{S_{a}f:f\in \mathcal{F}\}$ is a strongly pre-compact set in $\mathfrak{C}(E_{A}^{c})$, where
$E_{A}^{c}=\{x\in \mathbb{R}^n:|x|\leq A\}$ and $\mathfrak{C}(E_{a}^{c})$ denotes the continuous function space on $E^{c}_{A}$ with uniform norm.
By Ascoli-Arzel\`{a} theorem, it suffices to show that $\{S_{a}f: f\in \mathcal{F}\}$ is bounded and equicontinuous in $\mathfrak{C}(E_{A}^{c})$. In fact,
from Kolmogorov's inequality (see \cite[Lemma 2.8, p. 485]{GF1985}), we have
\begin{equation}\label{Kolm}
\|f\|_{L^{q}(Q,\frac{dx}{|Q|})}\leq C\|f\|_{L^{p,\infty}(Q,\frac{dx}{|Q|})}
\end{equation}
for any cube $Q$ and $0<q<p<\infty$. Applying H\"{o}lder's inequality and \eqref{Kolm} for $f\in \mathcal{F}$ and $x\in E_{A}^{c}$, we have
\begin{align*}
|S_{a}f(x)|&\leq \bigg\{\frac{1}{|B(0,a)|}\int_{|y|\leq a}|f(x+y)|^{q}dy\bigg\}^{1/q}\\
&= \bigg\{\frac{1}{|B(0,a)|}\int_{|y-x|\leq a}|f(y)|^{q}dy\bigg\}^{1/q}\\
&\leq C\|f\|_{L^{p,\infty}(\mathbb{R}^n)},
\end{align*}
where $1<q<p<\infty$ and the constant $C$ is independent of $f$ and $x$ here. On the other hand, for any $x_{1},x_{2}\in E_{A}^{c}$, by a direct computation, we obtain
\begin{equation}\label{EQUIC}
\begin{aligned}
|S_{a}f(x_{1})-S_{a}f(x_{2})|&\leq \frac{1}{|B(0,a)|}\int_{|y|\leq a}|f(x_{1}+y)-f(x_{2}+y)|dy\\
&\leq \bigg\{\frac{1}{|B(0,a)|}\int_{|y|\leq a}|f(x_{1}+y)-f(x_{2}+y)|^{q}dy\bigg\}^{1/q}\\
&\leq C\|f(\cdot+x_{2}-x_{1})-f(\cdot)\|_{L^{p,\infty}(\mathbb{R}^n)}.
\end{aligned}
\end{equation}
Thus, \eqref{WLP2} and \eqref{EQUIC} show the equicontinuity of $\{S_{a}f:f\in \mathcal{F}\}$.

Next we show that for small enough $a$, the set $\{S_{a}f:f\in \mathcal{F}\}$ is also a strongly pre-compact set in $L^{p,\infty}(\mathbb{R}^n)$. To do this, we need only to prove that the set $\{S_{a}f:f\in \mathcal{F}\}$ is a totally bounded set in $L^{p,\infty}(\mathbb{R}^n)$. Because the set
$\{S_{a}f:f\in \mathcal{F}\}$ is a totally bounded set in  $\mathfrak{C}(E_{A}^{c})$, hence for the above $\epsilon$ and $N$, there exist $\{f_{1},f_{2},\cdots,f_{m}\}\subset \mathcal{F}$, such that $\{S_{a}f_{1},S_{a}f_{2},\cdots, S_{a}f_{m}\}$ is a finite $\epsilon^{N+1}$-net in $\{S_{a}f:f\in \mathcal{F}\}$ in the norm of  $\mathfrak{C}(E_{A}^{c})$.
We then know that for any $f\in \mathcal{F}$, there is $1\leq j\leq m$ such that
\begin{equation}\label{EA3}
\sup_{y\in E^{c}_{A}}|S_{a}f(y)-S_{a}f_{j}(y)|<\epsilon^{N+1}.
\end{equation}
Below we show that $\{S_{a}f_{1}, S_{a}f_{2},\cdots, S_{a}f_{m}\}$ is also a finite $\epsilon-$net of $\{S_{a}f:f\in \mathcal{F}\}$ in the norm of $L^{p,\infty}(\mathbb{R}^n)$ if $a$ is small enough. For any $f\in \mathcal{F},$ there exists $f_{j} (1\leq j\leq m)$ such that
\begin{align*}
\|S_{a}f-S_{a}f_{j}\|_{L^{p,\infty}(\mathbb{R}^n)}
&= \|\big(S_{a}f-S_{a}f_{j}\big)\chi_{E_{A}}\|_{L^{p,\infty}(\mathbb{R}^n)}+\|\big(S_{a}f-S_{a}f_{j}\big)\chi_{E^{c}_{A}}\|_{L^{p,\infty}(\mathbb{R}^n)}\\
&=:I_{1}+I_{2}.
\end{align*}
We first give the estimate for $I_{1}$. By \eqref{a0} and the above $\epsilon$, there exists a constant $\delta>0$ such that if $a<\delta$, then
\begin{align*}
&\|\big(S_{a}f-f\big)\chi_{E_{A}}\|_{L^{p,\infty}(\mathbb{R}^n)}\leq \epsilon/8, \qquad
\|\big(S_{a}f_{j}-f_{j}\big)\chi_{E_{A}}\|_{L^{p,\infty}(\mathbb{R}^n)}\leq \epsilon/8.
\end{align*}
Applying \eqref{EA1} and \eqref{EA2} and the estimates above, we have
\begin{align*}
I_{1}&\leq \|\big(S_{a}f-f\big)\chi_{E_{A}}\|_{L^{p,\infty}(\mathbb{R}^n)}
+\|f\chi_{E_{A}}\|_{L^{p,\infty}(\mathbb{R}^n)}\\
&\qquad+\|f_{j}\chi_{E_{A}}\|_{L^{p,\infty}(\mathbb{R}^n)}
+\|\big(S_{a}f_{j}-f_{j}\big)\chi_{E_{A}}\|_{L^{p,\infty}(\mathbb{R}^n)}\\
&\leq \epsilon/2.
\end{align*}
For $I_{2}$, the inequalities \eqref{EA1} and \eqref{EA3} give us that
\begin{align*}
I_{2}&\leq A^{n/p}\sup_{y\in E^{c}_{A}}|S_{a}f(y)-S_{a}f_{j}(y)|\leq \epsilon/2.
\end{align*}
Therefore, we show that $\{S_{a}f_{1},S_{a}f_{2}\cdots,S_{a}f_{m}\}$ is also a finite $\epsilon-$net of $\{S_{a}f:f\in \mathcal{F}\}$ in the norm of $L^{p,\infty}(\mathbb{R}^n)$ if $a$ is small enough.

Finally, let us show that the set $\mathcal{F}$ is a relative compact set in $L^{p,\infty}(\mathbb{R}^n)$. Taking any sequence $\{f_{j}\}_{j=1}^{\infty}$ in $\mathcal{F}$, by the relative compactness of $\{S_{a}f:f\in \mathcal{F}\}$ in $L^{p,\infty}(\mathbb{R}^n)$, there exists a subsequence $\{S_{a}f_{j_{k}}\}_{k=1}^{\infty}$ of $\{S_{a}f_{j}:f_{j}\}$ that is convergent in $L^{p,\infty}(\mathbb{R}^n)$. So, for any $\epsilon>0$ there exists $K\in \mathbb{N}$ such that for any $k>K$ and $m\in \mathbb{N}$,
$$\|S_{a}f_{j_{k}}-S_{a}f_{j_{k+m}}\|_{L^{p,\infty}(\mathbb{R}^n)}<\epsilon.$$
By \eqref{a0}, we have
\begin{align*}
&\|f_{j_{k}}-f_{j_{k+m}}\|_{L^{p,\infty}(\mathbb{R}^n)}\\
&\leq \|S_{a}f_{j_{k}}-f_{j_{k}}\|_{L^{p,\infty}(\mathbb{R}^n)}
+\|S_{a}f_{j_{k}}-S_{a}f_{j_{k+m}}\|_{L^{p,\infty}(\mathbb{R}^n)}
+\|S_{a}f_{j_{k+m}}-f_{j_{k+m}}\|_{L^{p,\infty}(\mathbb{R}^n)}\\
&\leq 3\epsilon.
\end{align*}
This shows that the subsequence $\{f_{j_{k}}\}_{k=1}^{\infty}$ converges in $L^{p,\infty}(\mathbb{R}^n)$, since $L^{p,\infty}(\mathbb{R}^n)$ is a qusi-Banach space. Therefore, the set $\mathcal{F}$ is a relative compact set in $L^{p,\infty}(\mathbb{R}^n)$.

For $0<p\leq 1$, we only need to consider the case that $\mathcal{F}$ consists of only nonnegative functions. Denote $\mathcal{F}^{p/2}:=\{f^{p/2}:f\in \mathcal{F}\}$.
We claim that $\mathcal{F}^{p/2}$ is relatively compact in $L^{2,\infty}(\mathbb{R}^n)$. We only check condition $(ii)$ by
\begin{align*}
\|f^{p/2}(\cdot+y)-f^{p/2}(\cdot)\|_{L^{2,\infty}(\mathbb{R}^n)}&\leq \||f(\cdot+y)-f(\cdot)|^{p/2}\|_{L^{2,\infty}(\mathbb{R}^n)}\\
&=\|f(\cdot+y)-f(\cdot)\|^{p/2}_{L^{p,\infty}(\mathbb{R}^n)},
\end{align*}

Next, we prove that the relative compactness of $\mathcal{F}^{p/2}$ in $L^{2,\infty}(\mathbb{R}^n)$ implies the relative compactness of $\mathcal{F}$ in $L^{p,\infty}(\mathbb{R}^n)$.
For any sequence $\{f_{k}\}_{k=1}^{\infty}$ of $\mathcal{F}$, there exists a subsequence of $\{f^{p/2}_{k}\}_{k=1}^{\infty}$, denote by $\{f^{p/2}_{k_{i}}\}_{i=1}^{\infty}$,
which tends to $f^{p/2}$ in $L^{2,\infty}(\mathbb{R}^n)$. Using Lemma \ref{meas}, $f^{p/2}_{k_{i}}$ tends to $f^{p/2}$ locally in measure, we further choosing the diagonal subsequence of $\{f^{p/2}_{k_{i}}\}_{i=1}^{\infty}$, still denoted by $\{f^{p/2}_{k_{i}}\}_{i=1}^{\infty}$, pointwise tends to $f^{p/2}$ a.e. From this, $f_{k_{i}}\rightarrow f$ pointwise a.e., which further implies that $f_{k_{i}}\rightarrow f$ locally in measure.

On the other hand, since $\mathcal{F}^{p/2}$ is relatively compact in $L^{2,\infty}(\mathbb{R}^n)$, then $\mathcal{F}^{p/2}\subset UAC(L^{2,\infty}(\mathbb{R}^n))$ by Lemma \ref{RC}. One can easily verify that $\mathcal{F}^{p/2}\subset UAC(L^{2,\infty}(\mathbb{R}^n))$ implies $\mathcal{F}\subset UAC(L^{p,\infty}(\mathbb{R}^n))$.

Now, we have verified that $\mathcal{F}\subset UAC(L^{p,\infty}(\mathbb{R}^n))$ and $F$ is locally relative compact in
measure. The relative compactness of $\mathcal{F}$ follows by Lemma \ref{RC}.

\vspace{0.3cm}
To prove necessity, we first assume that condition $(i)$ in Theorem \ref{WLP} is violated. Then there exists a sequence $\{f_{m}\}$ of functions belonging to $\mathcal{F}$ such that the quasi-distance
$$\rho(f_{m},0)=\|f_{m}\|_{L^{p,\infty}(\mathbb{R}^n)}$$
tends to $+\infty$, By
$$\rho(f_{m},0)\leq C\big(\rho(f_{m},f)+\rho(f,0)\big),$$
we have $\rho(f_{m},f)\rightarrow +\infty$ as $m\rightarrow +\infty$. Hence, the set $\mathcal{F}$ is not compact.

We now assume that condition $(ii)$ does not hold. Then there exist $\delta>0$, a sequence
$$f_{1},f_{2},\cdots,f_{m},\cdots$$
of functions belonging to $\mathcal{F}$, and a sequence $a_{m}>0,$
$$\lim_{m\rightarrow +\infty}a_{m}=0$$
such that
$$\sup_{y\in B(O,a_{m})}\|f_{m}(\cdot+y)-f_{m}(\cdot)\|_{L^{p,\infty}(\mathbb{R}^n)}\geq \delta$$
for any $m$. Clearly
\begin{align*}
\delta\leq& \sup_{y\in B(O,a_{m})}\|f_{m}(\cdot+y)-f_{m}(\cdot)\|_{L^{p,\infty}(\mathbb{R}^n)}\\
\leq& C\sup_{y\in B(O,a_{m})}\|f_{m}(\cdot+y)-f(\cdot+y)\|_{L^{p,\infty}(\mathbb{R}^n)}\\
&+C\sup_{y\in B(O,a_{m})}\|f(\cdot+y)-f(\cdot)\|_{L^{p,\infty}(\mathbb{R}^n)}\\
&+C\|f_{m}(\cdot)-f(\cdot)\|_{L^{p,\infty}(\mathbb{R}^n)}\\
\leq &2C\|f_{m}(\cdot)-f(\cdot)\|_{L^{p,\infty}(\mathbb{R}^n)}+C\sup_{y\in B(O,a_{m})}\|f(\cdot+y)-f(\cdot)\|_{L^{p,\infty}(\mathbb{R}^n)}.
\end{align*}
Since
$$\lim_{m\rightarrow +\infty}\sup_{y\in B(O,a_{m})}\|f(\cdot+y)-f(\cdot)\|_{L^{p,\infty}(\mathbb{R}^n)}=0,$$
it follows that
$$\mathop{\underline{\lim}}\limits_{m\rightarrow +\infty}\|f_{m}(\cdot)-f(\cdot)\|_{L^{p,\infty}(\mathbb{R}^n)}\geq \frac{\delta}{2C}.$$
Consequently, the sequence $\{f_{m}\}$ and the set $\mathcal{F}$ are not compact.

Finally, if the condition $(iii)$ is not satisfied, there exist $\delta>0$, a sequence
$$f_{1},f_{2},\cdots,f_{m},\cdots$$
of functions belonging to $\mathcal{F}$, and a sequence $A_{m}>0,$
$$\lim_{m\rightarrow +\infty}A_{m}=+\infty,$$
such that
$$\|f_{m}\chi_{E_{A_{m}}}\|_{L^{p,\infty}(\mathbb{R}^n)}\geq \delta$$
for any $m$. Therefore,
\begin{align*}
\delta&\leq \|f_{m}\chi_{E_{A_{m}}}\|_{L^{p,\infty}(\mathbb{R}^n)}\\
&\leq C\|(f_{m}-f)\chi_{E_{A_{m}}}\|_{L^{p,\infty}(\mathbb{R}^n)}+C\|f\chi_{E_{A_{m}}}\|_{L^{p,\infty}(\mathbb{R}^n)}.
\end{align*}
Note that
$$\lim_{m\rightarrow +\infty}\|f\chi_{E_{A_{m}}}\|_{L^{p,\infty}(\mathbb{R}^n)}=0,$$
then
$$\|(f_{m}-f)\chi_{E_{A_{m}}}\|_{L^{p,\infty}(\mathbb{R}^n)}\geq \frac{\delta}{C}$$
and the set $\mathcal{F}$ is not compact. So the condition $(iii)$ must be hold.

We finish the proof of Theorem \ref{WLP}. \qed

\section{Compactness of commutators in Hardy type spaces}

The characterization of relative compactness in the classical $L^p(\mathbb{R}^n)$ Lebesgue spaces
was discovered by Kolmogorov (see \cite{Kol,Tik}) under some restrictive conditions.  Then it was extended by Riesz \cite{Rie}. The complete version of the classical
Riesz-Kolmogorov theorem can be stated as follows.
\begin{lemma}\label{LP-COMPACTNESS}
 (Classical Riesz-Kolmogorov theorem.)
Let $1\leq p<\infty$. A subset $\mathcal{F}$ of $L^{p}(\mathbb{R}^n)$ is relatively compact if and only if the following three conditions hold:
\begin{enumerate}
  \item[(i)] norm boundedness uniformly
  \begin{equation}\label{LP1}
  \sup_{f\in \mathcal{F}}\|f\|_{L^{p}(\mathbb{R}^n)}<\infty;
  \end{equation}
  \item[(ii)] translation continuty uniformly
  \begin{equation}\label{LP2}
\lim_{r\rightarrow 0}\sup_{y\in B(O,r)}\|f(\cdot+y)-f(\cdot)\|_{L^{p}(\mathbb{R}^n)}=0 \ \text{uniformly in} f\in \mathcal{F};
  \end{equation}
  \item[(iii)] control uniformly away from the origin
    \begin{equation}\label{LP3}
\lim_{\alpha \rightarrow \infty}\|f\chi_{E_{A}}\|_{L^{p}(\mathbb{R}^n)}=0 \  \text{uniformly in} f\in \mathcal{F},
  \end{equation}
  where $E_{A}=\{x\in \mathbb{R}^n:|x|>A\}$.
\end{enumerate}
\end{lemma}

As mentioned in the introduction, CMO is the closure in BMO of the space of $C^{\infty}$
functions with compact support. In \cite{U1}, it was shown that CMO can be characterized
in the following way.
\begin{lemma}\label{CMO}
 Let $f\in {\rm BMO}(\mathbb{R}^n).$ Then $f\in {\rm CMO}(\mathbb{R}^n)$ if and only if the following conditions hold:
\begin{enumerate}
  \item[(1)] $\lim_{\delta\rightarrow 0}\sup_{|Q|=\delta}\mathcal{O}(f;Q)=0$;
  \item[(2)] $\lim_{R\rightarrow \infty}\sup_{|Q|=R}\mathcal{O}(f;Q)=0$;
  \item[(3)] $\lim_{R\rightarrow \infty}\sup_{Q\cap[-d,d]^n=\empty}\mathcal{O}(f;Q)=0,$
\end{enumerate}
where
$$\mathcal{Q}(f;Q)=\frac{1}{|Q|}\int_{Q}|f(x)-f_{Q}|dx \qquad {\text and }\qquad f_{Q}=\frac{1}{|Q|}\int_{Q}f(x)dx.$$
\end{lemma}
Since then, there have been a lot of articles concerning the boundedness and the compactness of commutators on function spaces as well as their applications in PDEs, see \cite{CT2015,CH2015,CDW,CDW2009,DW2021, GWY1,AMP}. Krantz and Li in \cite{KL12001} and \cite{KL22001} have applied commutator theory to give a compactness characterization of Hankel operators on holomorphic Hardy spaces $H^{2}(D)$, where $D$ is a bounded, strictly pseudoconvex domain in $\mathbb{C}^n$. It is perhaps for this important reason that the boundedness of $[b,T]$ attracted one's attention among researchers in harmonic analysis and PDEs.

The compactness criteria were studied by many authors in various
settings. Meanwhile, it has played an important role in the compactness results of certain bounded operators in the field of harmonic
analysis. Let $\alpha\in [0,1]$. For the locally integral function $f$ and cube $Q$, we write
$$\mathcal{O}_{\alpha}(f;Q)=\frac{1}{|Q|^{1+\alpha/n}}\int_{Q}|f(x)-f_{Q}|dx.$$
Define by ${\rm CMO}_{\alpha}(\mathbb{R}^n)$ the $C^{\infty}_{c}(\mathbb{R}^n)$ closure in $Lip_{\alpha}(\mathbb{R}^n)$. The authors in \cite{GHWY} showed that
\begin{lemma}\label{CMO2}
Let $\alpha\in (0,1)$. A $Lip_{\alpha}(\mathbb{R}^n)$ function $f$ belongs to ${\rm CMO}_{\alpha}(\mathbb{R}^n)$ if it satisfies the following three conditions the following three conditions:
\begin{enumerate}
  \item[(1)] $\lim_{\delta\rightarrow 0}\sup_{|Q|=\delta}\mathcal{O}_{\alpha}(f;Q)=0$;
  \item[(2)] $\lim_{R\rightarrow \infty}\sup_{|Q|=R}\mathcal{O}_{\alpha}(f;Q)=0$;
  \item[(3)] $\lim_{R\rightarrow \infty}\sup_{Q\cap[-d,d]^n=\empty}\mathcal{O}_{\alpha}(f;Q)=0.$
\end{enumerate}
\end{lemma}

Now, we give the proofs of the compactness results in Hardy type spaces.

\vspace{0.3cm}
\textbf{Proof of Theorem \ref{thmHardyb}. }
Assume that $b\in {\rm CMO}(\mathbb{R}^n)$ and $E$ be a bounded set in $H_{b}^{1}(\mathbb{R}^n).$ It is enough to show that $[b,T](E)$ is relatively compact in $L^{1}(\mathbb{R}^n)$.

By the results in \cite{P1995}, the commutator $[b,T]$ maps form $H^{1}_{b}(\mathbb{R}^n)$ into $L^{1}(\mathbb{R}^n)$ with the estimates of the form
\begin{equation}\label{I-1}
\|[b,T](f)\|_{L^{1}(\mathbb{R}^n)}\lesssim \|b\|_{\rm BMO(\mathbb{R}^n)}\|f\|_{H^{1}_{b}(\mathbb{R}^n)}.
\end{equation}

On the other hand, since $b\in {\rm CMO}(\mathbb{R}^n)$, then for any $\epsilon>0$ there exists a function $b_{\epsilon}\in C^{\infty}_{c}(\mathbb{R}^n)$ such that
\begin{equation*}
\|b-b_{\epsilon}\|_{\rm BMO(\mathbb{R}^n)}<\epsilon.
\end{equation*}
P\'{e}rez in \cite{P1995} proved that the commutator $[b,T]$ of the function $b\in {\rm BMO}(\mathbb{R}^n)$
is bounded from $H^{1}_{b}(\mathbb{R}^n)$ to $L^{1}(\mathbb{R}^{n})$, which shows that
\begin{align*}
\|[b,T](f)\|_{L^{1}(\mathbb{R}^n)}&\lesssim \|[b-b_{\epsilon},T](f)\|_{L^{1}(\mathbb{R}^n)}+\|[b_{\epsilon},T](f)\|_{L^{1}(\mathbb{R}^n)}\\
&\lesssim \|b-b_{\epsilon}\|_{\rm BMO(\mathbb{R}^n)}\|f\|_{H^{1}_{b}(\mathbb{R}^n)}+\|[b_{\epsilon},T](f)\|_{L^{1}(\mathbb{R}^n)}\\
&\lesssim \epsilon+\|[b_{\epsilon},T](f)\|_{L^{1}(\mathbb{R}^n)}
\end{align*}
for all $f\in E$. Moreover, for any $b$-atom $a$, we have
\begin{align*}
\Big|\int_{\mathbb{R}^n}a(y)b_{\epsilon}(y)dy\Big|&=\Big|\int_{\mathbb{R}^n}a(y)\big(b_{\epsilon}(y)-b(y)\big)dy\Big|\\
&\leq \|a\|_{H^{1}(\mathbb{R}^n)}\|b-b_{\epsilon}\|_{\rm BMO(\mathbb{R}^n)}\leq \epsilon.
\end{align*}
Consequently, our task is to show that $b_{\epsilon},T$ is relatively compact in $L^{1}(\mathbb{R}^n)$. Given the definition of the space $H^{1}_{b}(\mathbb{R}^n)$, our proof hinges on demonstrating that for a $b-$atom $a$, the function $b_{\epsilon},T$ fulfills the conditions (2)-(3) as stipulated in Lemma \ref{LP-COMPACTNESS}.

Next, we show that $[b_{\epsilon},T](a)$ also satisfies (2). Indeed, suppose that ${\rm supp}(b_{\epsilon})\subset B_{R_{\epsilon}}$ for some $R_{\epsilon}>1$.
Then, for any $f\in E$ and $x\in B_{R}^{c}$ with $R>2R_{\epsilon}$, we get that
\begin{align*}
b_{\epsilon}(x)T(a)(x)=0
\end{align*}
and for some $y_{0}\in B(0,R_{\epsilon})$, we have
$$|x-y_{0}|\approx |x-y|\approx |x|$$
for any $y\in B_{R_{\epsilon}}$, then we obtain
\begin{align*}
&\big|[b_{\epsilon},T](a)(x)\big|=\big|T(b_{\epsilon}a)(x)\big|\\
&=\Big|\int_{\mathbb{R}^n}K(x,y)b_{\epsilon}(y)a(y)dy\Big|\\
&=\Big|\int_{\mathbb{R}^n}K(x,y)b_{\epsilon}(y)a(y)dy-\int_{\mathbb{R}^n}K(x,y_{0})b_{\epsilon}(y)a(y)dy\Big|\\
&\lesssim \int_{\mathbb{R}^n}\big|K(x,y)-K(x,y_{0})\big||b_{\epsilon}(y)||a(y)|dy\\
&\lesssim \frac{|y-y_{0}|^{\gamma}}{|x-y_{0}|^{n+\gamma}}\|a\|_{L^{\infty}(\mathbb{R}^n)}\|b_{\epsilon}\|_{L^{\infty}(\mathbb{R}^n)}|Q|\\
&\lesssim |x|^{-n-\gamma}.
\end{align*}
It follows that
\begin{align*}
\|[b_{\epsilon},T](a)\chi_{B_{R}^{c}}\|_{L^{1}(\mathbb{R}^n)}
&\lesssim \int_{|x|>R}|x|^{-n-\gamma}dx\lesssim R^{-\gamma}.
\end{align*}
This implies that
\begin{equation}\label{I-2}
\|[b_{\epsilon},T](a)(x)\chi_{B_{R}^{c}}(x)\|_{L^{1}(\mathbb{R}^n)}\rightarrow 0, \ \text{as} \ R\rightarrow \infty.
\end{equation}

Finally, we give the estimate for the condition (3). To do this, we prove that for any $\epsilon>0$, there exists a sufficiently small $|t|$ (independent
of $a$), then
\begin{equation}\label{Iepsilon}
\|[b_{\epsilon},T](a)(\cdot+t)-[b_{\epsilon},T](a)(\cdot)\|_{L^{1}(\mathbb{R}^n)}\lesssim \epsilon.
\end{equation}
We write
\begin{align*}
&[b_{\epsilon},T](a)(x+t)-[b_{\epsilon},T](a)(x)\\
&=\int_{\mathbb{R}^n}(b_{\epsilon}(x+t)-b_{\epsilon}(y))K(x+t,y)a(y)dy-\int_{\mathbb{R}^n}(b_{\epsilon}(x)-b_{\epsilon}(y))K(x,y)a(y)dy\\
&=\int_{|x-y|>\delta}(b_{\epsilon}(x+t)-b_{\epsilon}(x))K(x,y)a(y)dy\\
&\quad+\int_{|x-y|>\delta}(b_{\epsilon}(x+t)-b_{\epsilon}(x))\big(K(x+t,y)-K(x,y)\big)a(y)dy\\
&\quad+\int_{|x-y|\leq \delta}(b_{\epsilon}(y)-b_{\epsilon}(x))K(x,y)a(y)dy\\
&\quad+\int_{|x-y|\leq \delta}(b_{\epsilon}(x+t)-b_{\epsilon}(y))K(x+t,y)\big)a(y)dy\\
&=:I_{1}+I_{2}+I_{3}+I_{4},
\end{align*}
where, for a convenient choice of $\delta>0$ to be specified later. If we now let $T^*$ denote the maximal truncated bilinear singular integral operator
$$T^*(f)(x)=\sup_{\delta>0}\Big|\int_{|x-y|>\delta}K(x,y)f(y)dy\Big|,$$
then
\begin{align*}
|I_{1}|&\leq |b_{\epsilon}(x+t)-b_{\epsilon}(x)|\Big|\int_{|x-y|>\delta}K(x,y)a(y)dy\Big|\\
&\lesssim \|\nabla b_{\epsilon}\|_{L^{\infty}(\mathbb{R}^n)}|t|T^*(a)(x).
\end{align*}
In \cite{GT2002}, Grafakos and Torres proved that $T^{*}$ maps from $H^{1}(\mathbb{R}^n)$ into $L^{1}(\mathbb{R}^n)$. Then
\begin{equation}\label{4I1}
\begin{aligned}
\|I_{1}\|_{L^{1}(\mathbb{R}^n)}&\lesssim |t|.
\end{aligned}
\end{equation}
In order to estimate $I_{2}$, thanks to the smoothness of the kernel $K$ and the change of variables, we obtain
\begin{equation*}
\begin{aligned}
|I_{2}|&\lesssim \|b_{\epsilon}\|_{L^{\infty}(\mathbb{R}^n)}|t|^{\gamma}\int_{|x-y|>\delta}\frac{|a(y)|}{|x-y|^{n+\gamma}}dy\\
&\lesssim  \frac{|t|^{\gamma}}{\delta^{\gamma}}\|b_{\epsilon}\|_{L^{\infty}(\mathbb{R}^n)}M(a)(x)
\end{aligned}
\end{equation*}
and
\begin{equation}\label{4I2}
\begin{aligned}
\|I_{2}\|_{L^{1}(\mathbb{R}^n)}\lesssim \frac{|t|^{\gamma}}{\delta^{\gamma}}.
\end{aligned}
\end{equation}
To estimate the third term, we use the size estimate of the Calder\'{o}n-Zygmund
kernel $K$. We have
\begin{equation*}
\begin{aligned}
|I_{3}|&\lesssim \|\nabla b_{\epsilon}\|_{L^{\infty}(\mathbb{R}^n)}\int_{|x-y|\leq \delta}\frac{|a(y)|}{|x-y|^{n-1}}dy\\
&\lesssim  \|\nabla b_{\epsilon}\|_{L^{\infty}(\mathbb{R}^n)}\delta M(a)(x)
\end{aligned}
\end{equation*}
and
\begin{equation}\label{4I3}
\begin{aligned}
\|I_{3}\|_{L^{1}(\mathbb{R}^n)}\lesssim \delta.
\end{aligned}
\end{equation}
Similarly, we also obtain
\begin{equation}\label{4I4}
\begin{aligned}
\|I_{4}\|_{L^{1}(\mathbb{R}^n)}\lesssim \delta.
\end{aligned}
\end{equation}

Let us now define $t_{0}=\frac{\epsilon^{2}}{1+\|b_{\epsilon}\|_{L^{\infty}(\mathbb{R}^n)}+\|\nabla b_{\epsilon}\|_{L^{\infty}(\mathbb{R}^n)}}$. For each $0<|t|<t_0$ and select $\delta=|t|/\epsilon$. Inequalities \eqref{4I1}, \eqref{4I2}, \eqref{4I3} and \eqref{4I4} imply \eqref{Iepsilon}.
Combining this with the inequalities \eqref{I-1} and \eqref{I-2}, we conclude that $[b,T]$ is a compact. \qed

\vspace{0.3cm}

\textbf{Proof of Theorem \ref{thmHardy}. }
Assume that $b\in {\rm CMO}_{\alpha}(\mathbb{R}^n)$ and $E$ be a bounded set in $H^{1}(\mathbb{R}^n).$ It is enough to show that $[b,T](E)$ is relatively compact in $L^{q}(\mathbb{R}^n)$ with $q=\frac{n}{n-\alpha}>1$.

Since $b\in {\rm CMO}_{\alpha}(\mathbb{R}^n)$, then for any $\epsilon>0$ there exists a function $b_{\epsilon}\in C^{\infty}_{c}(\mathbb{R}^n)$ such that
\begin{equation*}
\|b-b_{\epsilon}\|_{Lip_{\alpha}(\mathbb{R}^n)}<\epsilon.
\end{equation*}
In \cite{LWY}, Lu, Wu and Yang proved that the commutator $[b,T]$ of the function $b\in Lip_{\alpha}(\mathbb{R}^n)$
is bounded from $H^{1}(\mathbb{R}^n)$ to $L^{\frac{n}{n-\alpha}}(\mathbb{R}^{n})$, which shows that
\begin{align*}
\|[b,T](f)\|_{L^{q}(\mathbb{R}^n)}&\lesssim \|[b-b_{\epsilon},T](f)\|_{L^{q}(\mathbb{R}^n)}+\|[b_{\epsilon},T](f)\|_{L^{q}(\mathbb{R}^n)}\\
&\lesssim \|b-b_{\epsilon}\|_{Lip_{\alpha}(\mathbb{R}^n)}\|f\|_{H^{1}(\mathbb{R}^n)}+\|[b_{\epsilon},T](f)\|_{L^{q}(\mathbb{R}^n)}\\
&\lesssim \epsilon+\|[b_{\epsilon},T](f)\|_{L^{q}(\mathbb{R}^n)}
\end{align*}
for all $f\in E$. Then, it suffices to demonstrate that $[b_{\epsilon},T](E)$ is relatively compact in $L^{q}(\mathbb{R}^n)$.
In addition, we obtain that $[b_{\epsilon},T](E)$ satisfies (1) in Lemma \ref{LP-COMPACTNESS}.

Next, we show that  $[b_{\epsilon},T](E)$ also satisfies (2). Indeed, suppose that ${\rm supp}(b_{\epsilon})\subset B_{R_{\epsilon}}$ for some $R_{\epsilon}>1$.
Then, for any $f\in E$ and $x\in B_{R}^{c}$ with $R>2R_{\epsilon}$, we get that
\begin{align*}
b_{\epsilon}(x)T(f)(x)=0 \qquad \text{and} \qquad |[b_{\epsilon},T](f)(x)|=|T(b_{\epsilon}f)(x)|.
\end{align*}
For $x\in B_{R}^{c}$ and $y\in B_{R_{\epsilon}}$, we get $|x-y|\approx |x|$ and
$$|[b_{\epsilon},T](f)(x)\chi_{B_{R}^{c}}(x)|\lesssim |x|^{-n}\chi_{B_{R}^{c}}(x)\|b_{\epsilon}f\|_{L^{1}(\mathbb{R}^n)}
\lesssim |x|^{-n}\chi_{B_{R}^{c}}(x)\|b_{\epsilon}\|_{\rm BMO(\mathbb{R}^n)}\|f\|_{H^{1}(\mathbb{R}^n)}.$$
It follows that
\begin{align*}
\|[b_{\epsilon},T](f)(x)\chi_{B_{R}^{c}}(x)\|_{L^{q}(\mathbb{R}^n)}&\leq \|[b_{\epsilon},T](f)(x)\chi_{B_{R}^{c}}(x)\|_{L^{q}(\mathbb{R}^n)}\\
&\lesssim \|f\|_{H^{1}(\mathbb{R}^n)}\bigg(\int_{|x|>R}|x|^{-nq}dx\bigg)^{1/q}\\
&\lesssim R^{-n+n/q}\|f\|_{H^{1}(\mathbb{R}^n)}.
\end{align*}
This implies that $\|[b_{\epsilon},T](f)(x)\chi_{B_{R}^{c}}(x)\|_{L^{q}(\mathbb{R}^n)}\rightarrow 0$, as $R\rightarrow \infty$.

To prove the condition (3), we prove that for every $\delta>0$, if $|t|$ is sufficiently small(merely depending on $\delta$), then for every $f\in E$,
\begin{equation}\label{delta}
\|[b_{\epsilon},T](f)(\cdot+t)-[b_{\epsilon},T](f)(\cdot)\|_{L^{q}(\mathbb{R}^n)}\lesssim \delta^{\eta},
\end{equation}
where $\eta=\min\{1-\alpha,\alpha+3\gamma\}.$ We write
\begin{align*}
&[b_{\epsilon},T](f)(x+t)-[b_{\epsilon},T](f)(x)\\
&=\int_{\mathbb{R}^n}(b_{\epsilon}(x+t)-b_{\epsilon}(y))K(x+t,y)f(y)dy-\int_{\mathbb{R}^n}(b(x)-b(y))K(x,y)f(y)dy\\
&=\int_{|x-y|>\delta^{-1}|t|}(b_{\epsilon}(x+t)-b_{\epsilon}(x))K(x,y)f(y)dy\\
&\quad+\int_{|x-y|>\delta^{-1}|t|}(b_{\epsilon}(x+t)-b_{\epsilon}(x))\big(K(x+t,y)-K(x,y)\big)f(y)dy\\
&\quad+\int_{|x-y|\leq \delta^{-1}|t|}(b_{\epsilon}(y)-b_{\epsilon}(x))K(x,y)f(y)dy\\
&\quad+\int_{|x-y|\leq \delta^{-1}|t|}(b_{\epsilon}(x+t)-b_{\epsilon}(y))K(x+t,y)\big)f(y)dy\\
&=:J_{1}+J_{2}+J_{3}+J_{4}.
\end{align*}
We first consider $J_{1}$. Thanks to the size of the kernel $K$, we get
\begin{equation}\label{5J1}
\begin{aligned}
|J_{1}|&\leq |b_{\epsilon}(x+t)-b_{\epsilon}(x)|\Big|\int_{|x-y|>\delta^{-1}|t|}K(x,y)f(y)dy\Big|\\
&\lesssim \|\nabla b_{\epsilon}\|_{L^{\infty}(\mathbb{R}^n)}|t|\int_{|x-y|>\delta^{-1}|t|}\frac{|f(y)|}{|x-y|^{n}}dy\\
&\lesssim \|\nabla b_{\epsilon}\|_{L^{\infty}(\mathbb{R}^n)}|t|\int_{|x-y|>\delta^{-1}|t|}\frac{|f(y)|}{|x-y|^{n-\alpha}}\cdot \frac{1}{|x-y|^{\alpha}}dy\\
&\lesssim \|\nabla b_{\epsilon}\|_{L^{\infty}(\mathbb{R}^n)}|t|(\delta^{-1}|t|)^{\alpha}I_{\alpha}(|f|)(x),
\end{aligned}
\end{equation}
where $I_{\alpha}$ stands for the fractional operator,
$$I_{\alpha}(f)(x)=\int_{\mathbb{R}^n}\frac{f(y)}{|x-y|^{n-\alpha}}dy.$$
By the $(H^{1}(\mathbb{R}^n),L^{q}(\mathbb{R}^n))$ boundedness of fractional integral operator $I_{\alpha}$, we obtain
\begin{equation}\label{4M1}
\begin{aligned}
\|J_{1}\|_{L^{q}(\mathbb{R}^n)}&\lesssim \|\nabla b_{\epsilon}\|_{L^{\infty}(\mathbb{R}^n)}|t|(\delta^{-1}|t|)^{\alpha}\|I_{\alpha}(f)\|_{L^{q}(\mathbb{R}^n)}\\
&\lesssim \|\nabla b_{\epsilon}\|_{L^{\infty}(\mathbb{R}^n)}|t|(\delta^{-1}|t|)^{\alpha}\|f\|_{H^{1}(\mathbb{R}^n)}\lesssim \delta^{-\alpha}|t|^{1+\alpha}.
\end{aligned}
\end{equation}
As for $J_2$, applying the smoothness of the kernel $K$, we deduce that
\begin{equation}\label{5J2}
\begin{aligned}
|J_{2}|&\lesssim \|b_{\epsilon}\|_{L^{\infty}(\mathbb{R}^n)}|t|^{\gamma}\int_{|x-y|>\delta^{-1}t}\frac{|f(y)|}{|x-y|^{n+\gamma}}dy\\
&\lesssim  \|b_{\epsilon}\|_{L^{\infty}(\mathbb{R}^n)}|t|^{\gamma}(\delta^{-1}|t|)^{\alpha+\gamma}I_{\alpha}(|f|)(x)
\end{aligned}
\end{equation}
and
\begin{equation}\label{4M2}
\begin{aligned}
\|J_{2}\|_{L^{q}(\mathbb{R}^n)}\lesssim \delta^{-\alpha-\gamma}|t|^{\alpha+2\gamma}.
\end{aligned}
\end{equation}

Next, we consider $J_{3}$. The H\"{o}lder inequality gives us that
\begin{equation}\label{5J3}
\begin{aligned}
|J_{3}|&\lesssim \|\nabla b_{\epsilon}\|_{L^{\infty}(\mathbb{R}^n)}\int_{|x-y|\leq \delta^{-1}|t|}\frac{|f(y)|}{|x-y|^{n-1}}dy\\
&\lesssim  \|\nabla b_{\epsilon}\|_{L^{\infty}(\mathbb{R}^n)}(\delta^{-1}|t|)^{1-\alpha}I_{\alpha}(|f|)(x)
\end{aligned}
\end{equation}
and
\begin{equation}\label{4M3}
\begin{aligned}
\|J_{3}\|_{L^{q}(\mathbb{R}^n)}\lesssim \delta^{-1+\alpha}|t|^{1-\alpha}.
\end{aligned}
\end{equation}
Similarly, we also obtain
\begin{equation}\label{5J4}
\begin{aligned}
|J_{4}|&\lesssim \|\nabla b_{\epsilon}\|_{L^{\infty}(\mathbb{R}^n)}\int_{|x-y|\leq \delta^{-1}|t|}\frac{|f(y)|}{|x+t-y|^{n-1}}dy\\
&\lesssim  \|\nabla b_{\epsilon}\|_{L^{\infty}(\mathbb{R}^n)}(\delta^{-1}|t|)^{1-\alpha}I_{\alpha}(|f|)(x+t)
\end{aligned}
\end{equation}
and
\begin{equation}\label{4M4}
\begin{aligned}
\|J_{4}\|_{L^{q}(\mathbb{R}^n)}\lesssim  \delta^{-1+\alpha}|t|^{1-\alpha}.
\end{aligned}
\end{equation}

A combination of the inequalities \eqref{4M1}, \eqref{4M2}, \eqref{4M3} and \eqref{4M4} provides us
$$\|[b_{\epsilon},T](f)(\cdot+t)-[b_{\epsilon},T](f)(\cdot)\|_{L^{q}(\mathbb{R}^n)}\lesssim  \delta^{-\alpha}|t|^{1+\alpha}+\delta^{-\alpha-\gamma}|t|^{\alpha+2\gamma}+\delta^{-1+\alpha}|t|^{1-\alpha}.$$
We assume that $|t|\leq \delta^{2}<1$, then
$$\|[b_{\epsilon},T](f)(\cdot+t)-[b_{\epsilon},T](f)(\cdot)\|_{L^{q}(\mathbb{R}^n)}\lesssim  \delta^{\eta},$$
where $\eta=\min\{1-\alpha,\alpha+3\gamma\}<1$. Thus we prove the inequality \eqref{delta} and $[b,T]$ is compact from $H^{1}(\mathbb{R}^n)$ to $L^{q}(\mathbb{R}^n)$. \qed

\section{Characterization of compactness of commutator in the endpoint space}

In this section, we will proceed with the proof of the following auxiliary lemmas,
which we need to prove our main results. We first recall a technical lemma about certain $H^\rho(\mathbb{R}^n)$ (see, \cite{WZ-JMAA} for $\rho=1$ and \cite{WZS-AFA} for $\rho<1$).
\begin{lemma} \label{lem3.1}
Let $\frac{n}{n+1}<\rho\le1$ and $f$ be a function satisfying the following estimates:
\begin{itemize}
  \item [\it (i)]  $\int_{\mathbb{R}^n}f(x)dx=0;$
  \item [\it (ii)] there exist balls $B_1=B(x_1,r)$ and $B_2=B(x_2,r)$ for some $x_1,x_2\in \mathbb{R}^n$ and $r>0$ such that
     $$|f(x)|\le h_1(x)\chi_{B_1}(x)+h_2(x)\chi_{B_2}(x),$$
     where $\|h_{i}\|_{L^{q}(\mathbb{R}^n)}\leq C|B_{i}|^{1/q-1/\rho}$ with $1<q\leq \infty$;
  \item [\it (iii)] $|x_1-x_2|= Nr$.
\end{itemize}
Then, $f\in H^\rho(\mathbb{R}^n)$ and there exists a positive constant $C$ independent of $x_1,x_2,r$ such that
$$ \|f\| _{H^{1} (\mathbb{R}^n)}\le C\log N ~~\text{for}~~ \rho=1 $$
and
$$ \|f\| _{H^{\rho}(\mathbb{R}^n)}\le CN^{n(\frac{1}{\rho}-1)} ~~\text{for}~~0<\rho<1. $$
\end{lemma}

\begin{lemma}\label{lem3.2}
Let $0<\alpha<1$, $1< q_0<q=\frac{n}{n-\alpha}$ and $\rho=\frac{n}{n+\alpha}$. For any $g\in L^\infty _{c}(\mathbb{R}^n ) , h\in L^\infty _{c}(\mathbb{R}^n )$, we have
\begin{equation}\label{hardy}
 \| {\textstyle \Pi}( g, h )\|_{H^{\rho}  ( \mathbb{R} ^n)}\lesssim
 \|   g   \|_{B^{{q_0}',{(1-\frac{q_0}{q})n}} ( \mathbb{R} ^n)}   \|  h\|_{ L^{1 }   ( \mathbb{R} ^n) },
\end{equation}
where $\Pi(g,h):=gT^{*}(h)-hT(g)$.
\end{lemma}
\begin{proof}
For any $g, h\in L^{\infty}_{c}(\mathbb{R}^n)$, to show that $\Pi(g,h)\in H^{\rho}(\mathbb{R}^n)$ with the norm \eqref{hardy}, we need to consider the properties of $\Pi(g,h)$.

Since $g,h$ are all in $L^{\infty}_{c}(\mathbb{R}^n)$, from the boundedness of Calder\'on-Zygmund operators, it is direct to see that $\Pi(g,h)\in L^{\rho}(\mathbb{R}^n)\bigcap L^{2}(\mathbb{R}^n)$ with compact support. Moreover, note that from the definition of $\Pi$, we have
$$\int_{\mathbb{R}^n}\Pi(g,h)(x)dx=0.$$
Hence, we immediately have that $\Pi(g,h)$ is a multiple of an $H^{\rho}(\mathbb{R}^n)$. Then it suffices to verify that $H^{\rho}(\mathbb{R}^n)$ norm of $\Pi(g,h)$ satifies \eqref{hardy}.

We first show that the inner product
\begin{equation}\label{def}
\langle b, \Pi(g,h)\rangle_{L^{2}(\mathbb{R}^n)}:=\int_{\mathbb{R}^n}b(x)\Pi(g,h)(x)dx
\end{equation}
is well defined for $b\in Lip_{\alpha}(\mathbb{R}^n)$.

Without loss of generality we assume that $\Pi(g,h)$ is supported in a cube $Q_{\Pi}$. We also note that for $b\in Lip_{\alpha}(\mathbb{R}^n) \cap L^{2}_{loc}(\mathbb{R}^n)$. As a consequence, we obtain
\begin{equation*}
\begin{aligned}
&\bigg|\int_{\mathbb{R}^n}b(x)\Pi(g,h)(x)dx\bigg|\\
&=|Q_{\Pi}|^{1+\alpha/n}\bigg|\frac{1}{|Q_{\Pi}|^{1+\alpha/n}}\int_{Q_{\Pi}}(b(x)-b_{Q_{\Pi}})\Pi(g,h)(x)dx\bigg|\\
&\leq |Q_{\Pi}|^{1+\alpha/n}\|b\|_{Lip_{\alpha}(\mathbb{R}^n)}\|\Pi(g,h)\|_{L^{2}(\mathbb{R}^n)}<\infty,
\end{aligned}
\end{equation*}
where the equality above follows from the cancellation condition of $\Pi(g,h)$ and the first inequality above follows from H\"{o}lder inequality. This shows that the inner product in \eqref{def} is well defined.

From Kolmogorov's inequality, we have $L^{q,\infty}(\mathbb{R}^n)\subset L^{q_{0},\beta}(\mathbb{R}^n)$ with $\beta=n(1-q_{0}/q)$.
Since $[b,T]$ is bounded from $L^{1}(\mathbb{R}^n)$ to $L^{q,\infty}(\mathbb{R}^n)$ when $b\in Lip_{\alpha}(\mathbb{R}^n)$, from the duality results of the block and Morrey spaces, we have
\begin{align*}
\Big| \int _{\mathbb{R}^n}b  ( x   ) {\textstyle \Pi}(g,h)dx \Big|&=\Big|\int _{\mathbb{R}^n} g  ( x   )    [ b,T]( h)   ( x   )  dx \Big|\\
&\le   \|g\|_{\mathcal{B}^{{q_0}',{\beta}}}\|[ b,T](h)\|_{L^{q_{0},\beta}(\mathbb{R}^n)}\\
&\le  C \|g\|_{\mathcal{B}^{{q_0}',{\beta}}}\|[ b,T](h)\|_{L^{q,\infty}(\mathbb{R}^n) }\\
&\lesssim  \| b   \|_{Lip_{\alpha}(\mathbb{R} ^n)}\|g\|_{\mathcal{B}^{{q_0}',{\beta}}}\|h\|_{ L^{1}(\mathbb{R}^n) }.
\end{align*}
We point out that from the fundamental fact as in \cite[Exercise 1.4.12(b)]{G2008}, we have ${\textstyle \Pi}  ( g, h )    ( x   )$ is in $H^\rho( \mathbb{R} ^n )$ with
\begin{align*}
\|{\textstyle \Pi}(g,h)\|_{H^{\rho}(\mathbb{R}^n)}
&\approx \sup_{b:\|b\|_{Lip_{\alpha}(\mathbb{R}^n)\leq 1}}\big|\langle b,\Pi(g,h)\rangle\big|\\
&\lesssim  \|g\|_{\mathcal{B}^{{q_0}',{(1-\frac{q_0}{q})n}}}\|h\|_{L^1( \mathbb{R} ^n) },
\end{align*}
which implies that \eqref{hardy} holds. The proof of Lemma \ref{lem3.2} is completed.
\end{proof}

\begin{lemma}\label{lem3.3}
Let $0<\alpha<1$, $1<q_{0}<q=\frac{n}{n-\alpha}$ and $\rho=\frac{n}{n+\alpha}$. If $f\in H^\rho(\mathbb{R}^n)$ can be written as
$$f=\sum_{k\ge 1}\lambda_k a_k.$$
Then, for any $\varepsilon>0$, there exists  $ \{ g^k \}_{k\ge1}, \{h^k\}_{k\ge1}\subset L_c^\infty ( \mathbb{R} ^n)$, and a large positive number $N$ (depending only on $\varepsilon $) such that
\begin{equation}\label{sec2-eq3.1}
\|a_{k}- {\textstyle \Pi}(g^{k}, h^{k})\|_{H^{\rho}(\mathbb{R}^n)}<\varepsilon
\end{equation}
and
\begin{equation*}
\sum_{k\ge1}|\lambda _k| \|g^k\|_{\mathcal{B}^{{q_0}',{n(1-\frac{q_{0}}{q})}}(\mathbb{R}^n)}   \|  h^k \|_{L^{1}(\mathbb{R}^n)}\le CN^{n}\|f\|_{H^\rho( \mathbb{R} ^n)}.
\end{equation*}
Furthermore, we have
\begin{equation*}
\Big\|f- \sum_{k\ge 1}\lambda _k{\textstyle \Pi}( g^k, h^k )\Big\|_{H^{\rho}(\mathbb{R}^n)}\le C\varepsilon\|f\|_{H^{\rho}( \mathbb{R} ^n)}.
\end{equation*}
\end{lemma}

\begin{proof}
Let $a$ be an $H^{\rho}(\mathbb{R} ^n)$-atom, supported in $ B(x_0, r)= :B_0$, for some $x_0\in \mathbb{R}^n$ and for $r>0$, such that
$$\int _{\mathbb{R} ^n}a  ( x   ) dx=0 \qquad and \qquad   \| a   \| _{L^\infty( \mathbb{R} ^n)}\le r^{-n+\alpha}. $$
We select $y_0 \in \mathbb{R}^n$ such that $|x_0-y_0|=Nr$. Apply the homogeneity of $T$, we get that for any $x\in B(y_0,r)$,
$$K(x_0-x)\ge \frac{C}{(Nr)^{n}}$$
and
\begin{align*}
 |   T^* (\chi _{B_{(y_0,r)}})   ( x_0  )   | &=\bigg|\int _{\mathbb{R} ^{n}}\frac{\chi _{B_{(y_0,r)}}( z ) }{| x_0-z |^{n} }dz \bigg|\\
&= \bigg|\int _{B(y_0,r)}\frac{1}{ | x_0-z |^{n} }dz\bigg|\\
&\geq C ( Nr  ) ^{-n} | B(y_0,r) |\\
& \geq C N^{-n}.
\end{align*}
Now, let us set
$$
g ( x  ) : = \chi _{B_{(y_0,r)} ( x  ) } \quad\text{and}\quad
h( x  ) : =  \frac{a ( x  )}{ T^\ast( g )( x_0)}.
$$
From the definitions of the functions $g(x)$  and $h(x) $, we arrive at
\begin{equation}\label{sec2-eq2.3}
\begin{cases}
\|  g   \|_{\mathcal{B}^{{q_0}',{(1-\frac{q_0}{q})n}}( \mathbb{R}^n) }\le Cr^{n-n/q}; \\
 \|  h\|_{ L^{1}(\mathbb{R}^n) }
\le \frac{Cr^n\|  a\|_{ L^\infty (\mathbb{R}^n) }}{ |  T ^*( g)   ( x_0  )|}\le CN^{n}r^{\alpha}.
\end{cases}
\end{equation}
By a direct computation, \eqref{sec2-eq2.3} shows that
\begin{equation}\label{sec2-eq2.4}
 \| g \|_{\mathcal{B}^{{q_0}',{(1-\frac{q_0}{q})n}}( \mathbb{R}^n) } \|  h\|_{ L^{1}(\mathbb{R}^n) }\le CN^{n}r^{n-\frac{n}{q}-\alpha}\le CN^{n}.
\end{equation}

Next, we have
\begin{align*}
a(x)- {\textstyle \Pi}( g, h)(x)&=a(x)- h  T^* (g)  ( x  ) +gT ( h) ( x  )\\
&=a ( x  )   \frac{T ^* (g)  ( x_0  )- T ^*(g)  ( x )}{T ^* (g)  ( x _0 )}+gT ( h) ( x  )\\
&=:{\rm I}_1 ( x  ) +{\rm I}_2 ( x  ) .
\end{align*}
Obviously, ${\rm I}_1( x  )$ is supported on $B_0$ and ${\rm I}_2 ( x )$ is supported on $B(y_0,r)$.

We first estimate ${\rm I}_1 ( x )$. For $x\in  B_0$, we have
\begin{align*}
|{\rm I}_1(x)|&=\Big|a(x)\frac{T^*(g)(x_0)-T ^* (g) ( x ) }{ T ^*(g) ( x _0 )}\Big|\\
&\leq C  \| a\|_{L^\infty }   N^{n}\Big|\int _{B(y_0,r)}\frac{ | x_0-x  |^{\gamma} }{  | z-x_0  |  ^{n-\alpha +\gamma} }dz\Big|\\
&\leq  C N^{n}r^{-n+\alpha}\frac{r^{n+\gamma}}{( Nr  )^{n+\gamma}}\\
&\leq \frac{C}{N^{\gamma}r^{n-\alpha}}.
\end{align*}
For the term ${\rm I}_2 ( x )$, it follows from the cancellation property of the atom $a$ that
\begin{align*}
|  T  ( h)  ( x )   |&\le\frac{1}{T ^*(g)   ( x_0)}\Big|\int _{B_0}\frac{  | x_0-x   |^{\gamma} }{| z-x_0|^{n-\alpha +\gamma} } a  ( z)dz\Big|\\
&\le CN^{n}\frac{r^{\gamma}}{  ( Nr   ) ^{n-\alpha +\gamma}} \int _{B_0}  | a  ( z )    | dz\\
&\le C\frac{  \| a   \| _{L^\infty }}{Nr^n} r^n\\
&\le \frac{C}{N^{\gamma}r^{n-\alpha}},
\end{align*}
As a consequence, we have
$$  |{\rm I}_2 ( x )    | \lesssim\frac{1}{Nr^{n-\alpha}}\chi _{B(y_0,r)} . $$
Combining the estimates of $ {\rm I}_1  ( x   )$ and $ {\rm I}_2 ( x )$, we obtain that
\begin{equation}\label{sec2-eqp2.3}
  |a- {\textstyle \Pi} ( g, h) ( x ) |\lesssim\frac{1}{Nr^{n-\alpha}} \chi _{B(x_{0},r)}+\frac{1}{Nr^{n-\alpha}}\chi _{B(y_0,r)}.
\end{equation}
In addition, we point out that
\begin{equation}\label{sec2-eqp2.4}
\int _{\mathbb{R}^n }  (a- {\textstyle \Pi} ( g, h) )dx=0,
\end{equation}
because the atom $a$ has cancellation property and the second integral equals $0$ just by the definitions of ${\textstyle \Pi}$. Then the inequality (\ref{sec2-eqp2.3}) and the cancellation (\ref{sec2-eqp2.4}), together with Lemma \ref{lem3.1} to the function $F(x)=a- {\textstyle \Pi}( g, h) ( x )$, we obtain
\begin{equation}\label{atom}
\|a- {\textstyle \Pi} ( g, h )    \|_{H^{\rho}  ( \mathbb{R}^n)} \le C\frac{\log N}{N^{\gamma}}.
\end{equation}
Let $N$ sufficiently large such that
\begin{equation}\label{atom2}
\frac{\log N}{N^{\gamma}}< \varepsilon.
\end{equation}
Therefore, we obtain (\ref{sec2-eq3.1}).

By applying \eqref{atom} and \eqref{atom2} to $a=a_{k}$ with $k\ge 1$, we obtain that there exist $\{g^{k}\}_{k\geq 1},\{h^{k}\}_{k\geq 1}\subset L^{\infty}_{c}(\mathbb{R}^n)$ such that
\begin{equation*}
\|a_{k}- {\textstyle \Pi} ( g^{k}, h^{k} )    \|_{H^{\rho}  ( \mathbb{R}^n)} \le C\varepsilon.
\end{equation*}
It follows from (\ref{sec2-eq2.4}) that
$$\| g ^k\|_{ B^{{q_0}',{(1-\frac{q_0}{q})n}}( \mathbb{R}^n) } \|  h^k\|_{ L^{1}(\mathbb{R}^n) }\le CN^{n}.$$
Thus,
\begin{equation*}
\sum_{k\ge1}|\lambda _k| \|   g^k\|_{B^{{q_0}',{(1-\frac{q_0}{q})n}}}  \|  h^k \|_{L^{1}( \mathbb{R} ^n)  }\le CN^{n}\|f\|_{H^\rho(\mathbb{R}^n)}.
\end{equation*}
This implies that
\begin{align*}
\Big\|f-{\textstyle \sum_{k\ge 1}} \lambda _k\textstyle \Pi( g^k, h^k ) \Big\|_{H^\rho(\mathbb{R}^n)}
&\le \sum_{k\ge 1}|\lambda _k|\Big\|a_k-\textstyle \Pi  ( g^k, h^k ) \Big\|_{H^\rho(\mathbb{R}^n)}\\
&\le C\varepsilon \sum_{k\ge1}|\lambda _k|\le C\varepsilon\|f\|_{H^\rho(\mathbb{R}^n)} .
\end{align*}
This ends the proof of Lemma \ref{lem3.3}.
\end{proof}

\begin{theorem}\label{thm1.1}
Suppose $1< q_0<q< \infty $, $\frac{n}{n+\gamma}<\rho<1 $ with $1-\frac{1}{q }-\frac{1}{\rho}+1=0$ and suppose that $T$ ia a Calder\'on-Zygmund operator that is homogeneous. Then for any $f \in {H^{\rho} ( \mathbb{R}^n   )}$ there exist sequence $ \{ \lambda _s^k  \} \in \ell ^\rho$ and functions $h_{s}^k\in L_c^\infty ( \mathbb{R} ^n)$, $g_s^k\in L_c^\infty ( \mathbb{R} ^n)$ such that
\begin{equation}\label{f1.1}
f = \sum_{k=1}^{\infty } \sum_{s=1}^{\infty }  \lambda _s^k  {\textstyle \Pi_{}^{}} ( g_s^k,h_{s}^k)
\end{equation}
in the sense of $H^\rho( \mathbb{R}^n   )$. Moreover
$$ \| f  \|_{H^\rho( \mathbb{R}^n   )}\approx  C\inf \bigg\{  (\sum_{s=1}^{\infty } \sum_{k=1}^{\infty }  | \lambda _s^k|^\rho\|   g_s^k \|_{\mathcal{B}^{{q_0}',{(1-\frac{q_0}{q})n}}( \mathbb{R} ^n  )}^\rho \|  h_{s}^k \|_{ L^1  ( \mathbb{R} ^n  ) }^\rho)^{\frac{1}{\rho}} \bigg\},$$
where the infimum above is taken over all possible representations of $f$ that satisfy \eqref{f1.1}.
\end{theorem}

\begin{proof}
By Lemma \ref{lem3.2}, it is obvious that
$$  \| {\textstyle \Pi}  ( g, h )    \|_{H^{\rho} (\mathbb{R}^n )}\lesssim  \|   g   \|_{ B^{{q_0}',{(1-\frac{q_0}{q})n}}(\mathbb{R}^n )}   \|  h \|_{ L^{1 }   (\mathbb{R}^n ) }. $$
It is immediate that for any representation of $f$ as in (\ref{f1.1}), i.e.,
$$f = \sum_{k=1}^{\infty } \sum_{s=1}^{\infty }  \lambda _s^k  {\textstyle \Pi}( g_s^k, h_{s}^k )( x   ), $$
with
$$ \| f  \|_{H^\rho ( \mathbb{R}^n   )}\le C\inf \{  \sum_{s=1}^{\infty } \sum_{k=1}^{\infty }  | \lambda _s^k   | ^\rho \|   g_s^k \|_{B^{{q_0}',{(1-\frac{q_0}{q})n}}( \mathbb{R} ^n  )} ^\rho \|  h_{s}^k \|_{ L^1  ( \mathbb{R} ^n  ) } ^\rho \}^{\frac{1}{\rho}} ,$$
where the infimum above is taken over all possible representations of $f$ that satisfy (\ref{f1.1}).
Next, utilizing the atomic decomposition, for any $f \in {H^{\rho} (\mathbb{R}^n )}$ we can find a sequence $  \{ \lambda _1^k   \}_{k\ge1} \in \ell ^\rho$ and sequence of $H^{\rho}(\mathbb{R}^n )$-atom  $\{ a_1^k   \}_{k\ge1} $ so that $f=\sum_{k=1}^{\infty } \lambda _1^ka_1^k$ and $(\sum_{k=1}^{ \infty }   | \lambda _1^k   |^\rho)^\frac{1}{\rho} \le C  \| f   \|_{ H^{\rho}(\mathbb{R}^n )}$.

Fix $\varepsilon > 0$ small enough such that $C\varepsilon < 1$. We apply Lemma \ref{lem3.3} to each atom $a_1^k$, then there exists $ \{g_1^k\}_{k\ge1}\in L^{\infty }_c (\mathbb{R}^n ) , \{h_{1}^k\}_{k\ge1} \in L^{\infty }_c (\mathbb{R}^n )$ such that
$$ \begin{cases}
 (\sum_{k\ge1}|\lambda ^k_1|^\rho\|   g^k_1\|_{B^{{q_0}',{(1-\frac{q_0}{q})n}}} ^\rho  \|  h^k_1 \|_{ L^{1}}^\rho)^{\frac{1}{\rho}} \le CN^{n}\|f\|_{H^\rho(\mathbb{R}^n )};\\
\|f- \sum_{k\ge 1}\lambda _k{\textstyle \Pi}  ( g^k_1, h^k_1 )    \|_{H^{\rho}(\mathbb{R}^n )}\le C\varepsilon\|f\|_{H^{\rho}(\mathbb{R}^n )}.
\end{cases}$$
Let us set
$$f_1=f-\sum_{k\ge 1}\lambda _k{\textstyle \Pi}  ( g^k_1, h^k_1 ) .$$
Since $f_1 \in H^\rho(\mathbb{R}^n )$, then we can decompose $f_1$ as follows:
$$f_1=\sum_{k=1}^{\infty } \lambda _2^ka_2^k,$$
where $  \{ \lambda _2^k   \}_{k\ge1} \in \ell ^\rho$, and $  \{ a _2^k   \}_{k\ge1} $ are atoms.
By applying Lemma \ref{lem3.3} to $f_1$, there exists $ \{g_2^k\}_{k\ge1}\in L^{\infty }_c (\mathbb{R}^n ) , \{h_{2}^k\}_{k\ge1} \in L^{\infty }_c (\mathbb{R}^n )$ such that
$$ \begin{cases}
 (\sum_{k\ge1}|\lambda _2^k|^\rho \|   g^k_2\|_{B^{{q_0}',{(1-\frac{q_0}{q})n}}} ^\rho \|  h^k_2 \|_{ L^{1} }^\rho)^{\frac{1}{\rho}}\le CN^{n}\|f_1\|_{H^\rho(\mathbb{R}^n )}\le C\varepsilon N^{n} \|f\|_{H^\rho(\mathbb{R}^n )};\\
\|f_1- \sum_{k\ge 1}\lambda ^k_2{\textstyle \Pi}  ( g^k_2, h^k_2 )    \|_{H^{\rho}(\mathbb{R}^n )}\le C\varepsilon\|f_1\|_{H^{\rho}}\le C\varepsilon^2\|f\|_{H^{\rho}(\mathbb{R}^n )}.
\end{cases}$$
Similarly, we can repeat the above argument to
\begin{align*}
f_2&=f_1- \sum_{k\ge 1}\lambda ^k_2{\textstyle \Pi}  ( g^k_2, h^k_2 )\\
&=f-\sum_{k\ge 1}\lambda _k{\textstyle \Pi}  ( g^k_1, h^k_1 )-\sum_{k\ge 1}\lambda ^k_2{\textstyle \Pi}  ( g^k_2, h^k_2 ).
\end{align*}
In summary, we can construct a sequence $  \{ \lambda _s^k   \}_{k\ge1} \in \ell ^\rho$, $ \{g_s^k\}_{k\ge1}\in L^{\infty }_c (\mathbb{R}^n ) , \{h_{s}^k\}_{k\ge1} \in L^{\infty }_c (\mathbb{R}^n )$, such that
$$ \begin{cases}
 (\sum_{s= 1}^{N}\sum_{k\ge1}|\lambda _s^k|^\rho \|   g^k_s\|_{B^{{q_0}',{(1-\frac{q_0}{q})n}}}  ^\rho \|  h^k_s\|_{ L^{1} }^\rho)^{\frac{1}{\rho}}\le C N^{n} \sum_{s= 1}^{N}\varepsilon ^{s-1} \|f\|_{H^\rho(\mathbb{R}^n )};\\
f=\sum_{s= 1}^{N} \sum_{k\ge 1}\lambda ^k_s{\textstyle \Pi}  ( g^k_s, h^k_s )   +f_N;\\
\|f_N\|_{H^\rho(\mathbb{R}^n )}\le C\varepsilon ^N\|f\|_{H^\rho(\mathbb{R}^n )}.
\end{cases}$$
Thus, the desired result follows as $N \to \infty$. This puts an end to the proof of Theorem \ref{thm1.1}.
\end{proof}

As a direct application, we give the characterization of Lipschitz spaces via commutators of Calder\'on-Zygmund and fractional integral operator.

\begin{theorem}\label{thm4.1}
Let $0<\alpha< 1$. The commutator $ [ b, T ] $ is bounded from $L^{1} (\mathbb{R}^n) \to L^{\frac{n}{n-\alpha},\infty} ( \mathbb{R}^n  )$ if and only if $b\in Lip_\alpha(\mathbb{R}^n)$.
\end{theorem}
\begin{proof}
The upper bound in this theorem is obvious. For the lower bound, suppose that $f\in H^p(\mathbb{R}^n)$ with $\frac{1}{p}-1=\frac{\alpha}{n}$, using the weak factorization in Theorem \ref{thm1.1} we obtain
\begin{align*}
 \langle b, f  \rangle   _{L^2 ( \mathbb{R}^n   ) }&= \langle b, \sum_{k=1}^{\infty } \sum_{s=1}^{\infty }  \lambda _s^k  {\textstyle\Pi} ( g_s^k,h_{s}^k)\rangle   _{L^2 ( \mathbb{R}^n   ) }\\
&=\sum_{k=1}^{\infty }\sum_{s=1}^{\infty } \lambda _s^k \langle b, {\textstyle\Pi}  ( g_s^k, h_{s}^k) \rangle  _{L^2 ( \mathbb{R}^n   ) }\\
&=\sum_{k=1}^{\infty }\sum_{s=1}^{\infty } \lambda _s^k \langle g_s^k, [ b,T]( h_{s}^k )  \rangle  _{L^2 ( \mathbb{R}^n   ) }.
\end{align*}
By the boundedness of $[ b,T]$ and the dual theorem between $M^{q}_{q_0} ( \mathbb{R} ^n)$ and $B^{{q_0}',{(1-\frac{q_0}{q})n}}(\mathbb{R}^n)$, we get
\begin{align*}
 |  \langle b, f  \rangle _{L^2 ( \mathbb{R}^n) }  |&\le \sum_{k=1}^{\infty }\sum_{s=1}^{\infty } |   \lambda _s^k |
 \|   g_s^k  \|_{{B^{{q_0}',{(1-\frac{q_0}{q})n}}(\mathbb{R}^n)}} \| [ b, T   ]  ( h_{s}^k)  \|_{M^{q}_{q_0} ( \mathbb{R} ^n)}\\
 &\le C\sum_{k=1}^{\infty }\sum_{s=1}^{\infty } |   \lambda _s^k |
 \|   g_s^k  \|_{{B^{{q_0}',{(1-\frac{q_0}{q})n}}(\mathbb{R}^n)}} \| [ b, T   ]  ( h_{s}^k)  \|_{L^{q,\infty} ( \mathbb{R} ^n)}\\
&\lesssim  \|  [ b, T   ] : L^{1} (\mathbb{R} ^n) \to L^{q ,\infty } ( \mathbb{R} ^n )  \| \\
&\qquad\times\sum_{k=1}^{\infty }\sum_{s=1}^{\infty } |   \lambda _s^k |
 \|   g_s^k \|_{B^{{q_0}',{(1-\frac{q_0}{q})n}}( \mathbb{R} ^n  )}  \|  h_{s}^k \|_{ L^1  ( \mathbb{R} ^n  ) }\\
&\lesssim  \|  [ b, T   ] : L^{1} (\mathbb{R} ^n) \to L^{q ,\infty } ( \mathbb{R} ^n ) \|  \| f \| _{H^p( \mathbb{R} ^n) }.
\end{align*}
By the duality between $H^p( \mathbb{R} ^n  )$ and $Lip_\alpha( \mathbb{R} ^n  ) $ with $\frac{1}{p}-1=\frac{\alpha}{n}$, we have that
\begin{align*}
\|b\|_{Lip_\beta ( \mathbb{R} ^n)}&\approx \underset{\|f\|_{H^p(\mathbb{R}^n)}\le1}{\sup} |\langle b, f  \rangle_{L^2 ( \mathbb{R}^n) } |\lesssim \|  [ b, T   ] : L^{1} (\mathbb{R} ^n) \to L^{q ,\infty } ( \mathbb{R} ^n )\|.
\end{align*}
Hence, we complete the proof of Theorem \ref{thm4.1}.
\end{proof}

Now, we turn to prove the Theorem \ref{thmL1}.

\textbf{Proof of Theorem \ref{thmL1}.}
{\bf Necessity.} Assume that $b\in {\rm CMO}_{\alpha}(\mathbb{R}^n)$ and $E$ be a bounded set in $L^{1}(\mathbb{R}^n).$ It is enough to show that $[b,T](E)$ is relatively compact in $L^{q,\infty}(\mathbb{R}^n)$.
Similar to the proof of Theorem \ref{thmHardy}, it suffices to demonstrate that $[b_{\epsilon},T](E)$ is relatively compact in $L^{q,\infty}(\mathbb{R}^n)$.

Note that $[b_{\epsilon},T](f)(x)\leq \|b_{\epsilon}\|_{Lip_{\alpha}(\mathbb{R}^n)}I_{\alpha}(|f|)(x)$ for any $x\in \mathbb{R}^n$, we obtain that $[b_{\epsilon},T](E)$ satisfies the condition (1) in Lemma \ref{CMO2}.

For the condition (2), suppose that ${\rm supp}(b_{\epsilon})\subset B_{R_{\epsilon}}$ for some $R_{\epsilon}>1$.
Then, for any $f\in E$ and $x\in B_{R}^{c}$ with $R>2R_{\epsilon}$, we get that $b_{\epsilon}(x)T(f)(x)=0$ and
\begin{align*}
&|[b_{\epsilon},T](f)(x)|=|T(b_{\epsilon}f)(x)|\\
&\leq C_{0}\|b_{\epsilon}\|_{L^{\infty}(\mathbb{R}^n)}\int_{|y|<R_{\epsilon}}\frac{|f(y)|}{|x-y|^{n}}dy.
\end{align*}
For $x\in B_{R}^{c}$ and $y\in B_{R_{\epsilon}}$, we get $|x-y|\approx |x|$ and
$$|[b_{\epsilon},T](f)(x)\chi_{B_{R}^{c}}(x)|\lesssim |x|^{-n}\chi_{B_{R}^{c}}(x)\|f\|_{L^{1}(\mathbb{R}^n)}.$$
It follows from $L^{q}(\mathbb{R}^n)\subset L^{q,\infty}(\mathbb{R}^n)$ that
\begin{align*}
\|[b_{\epsilon},T](f)(x)\chi_{B_{R}^{c}}(x)\|_{L^{q,\infty}(\mathbb{R}^n)}&\leq \|[b_{\epsilon},T](f)(x)\chi_{B_{R}^{c}}(x)\|_{L^{q}(\mathbb{R}^n)}\\
&\lesssim \|f\|_{L^{1}(\mathbb{R}^n)}\bigg(\int_{|x|>R}|x|^{-qn}dx\bigg)^{1/q}\\
&\lesssim R^{-\frac{n}{q'}}\|f\|_{L^{1}(\mathbb{R}^n)}.
\end{align*}
This implies that $\|[b_{\epsilon},T](f)(x)\chi_{B_{R}^{c}}(x)\|_{L^{q,\infty}(\mathbb{R}^n)}\rightarrow 0$, as $R\rightarrow \infty$.

Next, we give the estimate for the condition (3). For every $\delta>0$, if $|t|$ is sufficiently small(merely depending on $\delta$), we have
$$[b_{\epsilon},T](f)(x+t)-[b_{\epsilon},T](f)(x)=J_{1}+J_{2}+J_{3}+J_{4},$$
where the precise definitions of $J_{i},i=1,2,3,4$ are given in the above section.
The inequalities \eqref{5J1}, \eqref{5J2}, \eqref{5J3} and \eqref{5J4} show that
$$\|[b_{\epsilon},T](f)(\cdot+t)-[b_{\epsilon},T](f)(\cdot)\|_{L^{q,\infty}(\mathbb{R}^n)}\lesssim  \delta^{-\alpha}|t|^{1+\alpha}+\delta^{-\alpha-\gamma}|t|^{\alpha+2\gamma}+\delta^{-1+\alpha}|t|^{1-\alpha}.$$
If $|t|\leq \delta^{2}<1$, then
$$\|[b_{\epsilon},T](f)(\cdot+t)-[b_{\epsilon},T](f)(\cdot)\|_{L^{q,\infty}(\mathbb{R}^n)}\lesssim  \delta^{\eta},$$
where $\eta=\min\{1-\alpha,\alpha+3\gamma\}$. Thus $[b,T]$ is a compact operator maps from $L^{1}(\mathbb{R}^n)$ inon $L^{q,\infty}(\mathbb{R}^n)$.

Now, we are ready to demonstrate that $b\in {\rm CMO}_{\alpha}(\mathbb{R}^n)$. Seeking a contradiction, we assume that $b\notin {\rm CMO}_{\alpha}(\mathbb{R}^n)$. Therefore, $b$ violates (1), (2) and (3) in Lemma \ref{CMO2}.

{\bf Case 1.} If (1) does not hold true for the function $b$, then there exists a sequence of balls $\{B_{k}=B(x_{k},\delta_{k})\}_{k\geq 1}$ such that $\delta_{k}\rightarrow 0$ as $k\rightarrow \infty$, and
\begin{equation}\label{5N1}
\frac{1}{|B_{k}|^{1+\alpha/n}}\int_{B_{k}}|b(x)-b_{B_{k}}|dx\geq c_{0}>0,
\end{equation}
for every $k\geq 1.$ Without loss of generality, we assume that the sequence of $\{\delta_{k}\}_{k\geq 1}$ such that
$$C\delta_{k+1}\leq \delta_{k},$$
for some $C>1$ and all $k\geq 1$. Define $m_{b}(\Omega)$,the median value of function $b$ on a bounded set $\Omega\subset \mathbb{R}^n$, by
\begin{equation*}
\begin{cases}
|\{x\in \Omega: b(x)>m_{b}(\Omega)\}|\leq \frac{1}{2}|\Omega|, \\
|\{x\in \Omega: b(x)<m_{b}(\Omega)\}|\leq \frac{1}{2}|\Omega|.
\end{cases}
\end{equation*}
Let $y_{k}\in \mathbb{R}^n$ be such that $|x_{k}-y_{k}|=M\delta_{k}, M>10$ for any $k\geq 1$. Set
$$\tilde{B}_{k}=B(y_{k},\delta_{k}), \tilde{B}_{k,1}=\Big\{y\in\tilde{B}_{k}:b(y)\leq m_{b}(\tilde{B}_{k})\Big\}, \tilde{B}_{k,2}=\Big\{y\in\tilde{B}_{k}:b(y)\geq m_{b}(\tilde{B}_{k})\Big\},$$
and
$$B_{k,1}=\Big\{x\in B_{k}:b(x)\geq m_{b}(\tilde{B}_{k})\Big\}, B_{k,2}=\Big\{x\in B_{k}:b(x)\leq m_{b}(\tilde{B}_{k})\Big\}.$$
Also we write
\begin{equation*}
F_{k,1}=\tilde{B}_{k,1}\backslash \bigcup _{j=k+1}^{\infty}\tilde{B}_{j,1}, F_{k,2}=\tilde{B}_{k,2}\backslash \bigcup _{j=k+1}^{\infty}\tilde{B}_{j,2}.
\end{equation*}
Note that $F_{k,1}\cap F_{j,1}=\emptyset$ for $j\neq k$, and
\begin{equation*}\label{5N2}
\delta_{k}^{n}\gtrsim |F_{k,1}|\geq |\tilde{B}_{k,1}|-\sum_{j=k+1}^{\infty}|\tilde{B}_{j}|\gtrsim \delta_{k}^n-\sum_{j=k+1}^{\infty}\delta_{j}^n\gtrsim (1-\frac{1}{C-1})\delta^{n}_{k}.
\end{equation*}
Thus, $|F_{k,1}|\thickapprox |\tilde{B}_{k}|$. By the same analogue above, we also have
\begin{equation*}\label{5N3}
|F_{k,2}|\thickapprox |\tilde{B}_{k}|.
\end{equation*}
From the construction, we have
\begin{equation}\label{5N4}
\Big|b(x)-m_{b}(\tilde{B})\Big|\leq |b(x)-b(y)|, \forall (x,y)\in B_{k,l}\times \tilde{B}_{k,l}, l=1,2.
\end{equation}
Next, it follows from the triangle inequality and inequality \eqref{5N1} that
\begin{equation}\label{5N5}
\begin{aligned}
c_{0}&\leq \frac{1}{|B_{k}|^{1+\alpha/n}}\int_{B_{k}}|b(x)-b_{B_{k}}|dx\\
&\leq \frac{2}{|B_{k}|^{1+\alpha/n}}\int_{B_{k}}|b(x)-m_{b}(\tilde{B}_{k})|dx\\
&\leq \frac{2}{|B_{k}|^{1+\alpha/n}}\int_{B_{k,1}}|b(x)-m_{b}(\tilde{B}_{k})|dx\\
&\qquad+\frac{2}{|B_{k}|^{1+\alpha/n}}\int_{B_{k,2}}|b(x)-m_{b}(\tilde{B}_{k})|dx\\
&=:M_{1}+M_{2}.
\end{aligned}
\end{equation}
Then for any $k\geq 1$, $M_{1}\geq \frac{c_{0}}{2}$ or $M_{2}\geq \frac{c_{0}}{2}$. Thus, one can assume without loss of generality that $M_{1}\geq \frac{c_{0}}{2}$.

Let $\phi_{k}(x)=|B_{k}|^{-1}\big(\chi_{F_{k,1}}(x)-\frac{|F_{k,1}|}{|\tilde{B}_{k}|}\big)\chi_{\tilde{B}_{k}}(x)$. It is easy to check that $\phi_{k}$ satisfies
$${\rm supp} \phi_{k}\subset \tilde{B}_{k}, \int_{\mathbb{R}^n}\phi_{k}(x)dx=0$$
and $\|\phi_{k}\|_{L^{1}}\lesssim 1$. Moreover, $\phi_{k}\in H^{\frac{1}{2}}$ with $\|\phi_{k}\|_{H^{1/2}}\lesssim |B_{k}|$.

By the homogonous of $T$, we know that for any $k\geq 1$ and $x\in B_{k}$, one has
$$\frac{1}{M^{n}}\lesssim |T(\phi_{k})(x)|.$$
Furthermore, $T(\phi_{k})(x)$ is a constant sign in $\tilde{B}_{k}$. It follows from the inequalities \eqref{5N4} and \eqref{5N5} that
\begin{equation}\label{5N6}
\begin{aligned}
\frac{c_{0}}{2M}&\leq \frac{1}{M|B_{k}|^{1+\alpha/n}}\int_{B_{k,1}}|b(x)-m_{b}(\tilde{B}_{k})|dx\\
&\lesssim \frac{1}{|B_{k}|^{\alpha/n}}\int_{B_{k,1}}|b(x)-m_{b}(\tilde{B}_{k})||T(\phi_{k})(x)|dx\\
&=\frac{1}{|B_{k}|^{\alpha/n}}\int_{B_{k,1}}\Big|\int_{\mathbb{R}^n}(b(x)-m_{b}(\tilde{B}_{k}))K(x,y)\phi_{k}(y)dy\Big|dx\\
&\leq \frac{1}{|B_{k}|^{\alpha/n}}\int_{B_{k,1}}\Big|\int_{\mathbb{R}^n}(b(x)-b(y))K(x,y)\phi_{k}(y)dy\Big|dx\\
&=\frac{1}{|B_{k}|^{\alpha/n}}\int_{B_{k,1}}\Big|[b,T](\phi_{k})(x)\Big|dx.
\end{aligned}
\end{equation}
By Komogrove inequality and $1=1/q+\alpha/n$, we have
$$\frac{c_{0}}{2M}\lesssim \|[b,T](\chi_{F_{k,1}})\|_{L^{q,\infty}}.$$
The inequality \eqref{5N6} and the boundedness of $[b, T]$ from $L^{1}(\mathbb{R}^n)$ to $L^{q,\infty}(\mathbb{R}^n)$ give us that
\begin{equation}\label{5N7}
\|[b,T](\phi_{k})\|_{L^{q,\infty}}\approx 1.
\end{equation}

On the other hand, since $\big|[b,T](f)(x)\big|\lesssim I_{\alpha}(|f|)(x)$ and $I_{\alpha}$ maps from $H^{\frac{1}{2}}(\mathbb{R}^n)$ into $L^{\frac{n}{2n-\alpha}}(\mathbb{R}^n)$,
we have
\begin{align*}
\|[b,T](\phi_{k})\|_{L^{\frac{n}{2n-\alpha}}}&\lesssim \|b\|_{Lip_{\alpha}(\mathbb{R}^n)}\|\phi_{k}\|_{H^{1/2}}\\
&\lesssim \|b\|_{Lip_{\alpha}(\mathbb{R}^n)} \delta_{k}^{n}.
\end{align*}
Thus, $[b,T](\phi_{k})\rightarrow 0$ in $L^{\frac{n}{2n-\alpha}}$, as $k\rightarrow \infty$. This contradicts \eqref{5N7}. In other words, $b$ must satisfy (1). Similarly, we also obtain the desired result if $M_{1}\geq c_{0}/2$ holds true. In conclusion, $b$ cannot violate (1).

{\bf Case 2.} Assume that $b$ violates (2).

The same argument as in the proof of {\bf Case 1} by considering $R_{k}$ in place of $\delta_{k}$, with $R_{k}\rightarrow \infty$.  By repeating the above proof for $R_{k}$ in place of $\delta_{k}$, we also obtain \eqref{5N7}. For any $p>1$, let $\tilde{q}$ with $1/\tilde{q}=1/p-\alpha/n$. Since $[b,T]$ maps $L^{p}$ to $L^{\tilde{q}}$, then we have
\begin{align*}
\|[b,T](\phi_{k})\|_{L^{\tilde{q}}}&\lesssim \|b\|_{Lip_{\alpha}(\mathbb{R}^n)}\|\phi_{k}\|_{L^{p}}\\
&\lesssim \|b\|_{Lip_{\alpha}(\mathbb{R}^n)} \delta_{k}^{-\frac{n}{p'}}.
\end{align*}
Thus, $[b,T](\chi_{k})\rightarrow 0$ in $L^{\tilde{q}}$, when $k\rightarrow \infty$. As a result, $b$ satisfies (2).

{\bf Case 3.} The proof of this case is similar to the one of {\bf Case 2}. Thus, we leave it to the reader.
From the above cases, we conclude that $b\in {\rm CMO}_{\alpha}(\mathbb{R}^n)$. \qed

\section{Appendix}

Now, we show that commutators of singular integrals with ${\rm CMO}(\mathbb{R}^n)$ functions are not compact in $L^{1,\infty}(\mathbb{R}^n)$.
\begin{proposition}\label{C-E}
There exists a function $b\in {\rm CMO}(\mathbb{R}^n)$ such that $[b,T]$ is not a compact operator from $L\log L(\mathbb{R}^n)$ to $L^{1,\infty}(\mathbb{R}^n)$.
\end{proposition}
\begin{proof}
Without loss of generality, we only deal with $n = 1$ and $T=H$, where $H$ is the Hilbert transform
$$H(f)(x)=p.v.\int_{\mathbb{R}}\frac{f(y)}{x-y}dy.$$
Let $f=-\chi_{(-1,1)}$ and $b\in C^{\infty}_{c}(\mathbb{R})$ such that
\begin{equation*}
  b(x)= \left\{
   \begin{array}{ll}\vspace{1ex}
0, & |x|> 2,\\
1, & |x|< 1.
   \end{array}
 \right.
\end{equation*}
Then $f\in L\log L(\mathbb{R})$ and $b\in {\rm CMO}(\mathbb{R})$. For any $|x|>2$, we have
\begin{equation}\label{C1}
\begin{aligned}
[b,H][b,H](f)(x)&=\int_{\mathbb{R}}\frac{b(x)-b(y)}{x-y}f(y)dy
=\int_{\mathbb{R}}\frac{0-1}{x-y}f(y)dy\\
&=\int_{-1}^{1}\frac{1}{x-y}dy=\log \big(1+\frac{2}{x-1}\big).
\end{aligned}
\end{equation}
By \eqref{C1}, we conclude that for any $A>2$ and $0<\lambda<\min \{\log \big(1+\frac{2}{A-1}\big), -\log \big(1-\frac{2}{A+1}\big)\}$,
\begin{equation}\label{6p-1}
\begin{aligned}
&\Big|\{x\in \mathbb{R}: |[b,H](f)(x)\chi_{E_{A}}(x)|>\lambda\}\Big|\\
&=\Big|\{x\in \mathbb{R}: |\log \big(1+\frac{2}{x-1}\big)|\chi_{E_{A}}(x)>\lambda\}\Big|\\
&=\Big|\{x>A: \log \big(1+\frac{2}{x-1}\big)>\lambda\}\Big|\\
&\qquad+\big|\{x<-A: -\log \big(1+\frac{2}{x-1}\big)>\lambda\}\Big|.
\end{aligned}
\end{equation}
By a direct computation, we obtain
\begin{equation}\label{6p-2}
\begin{aligned}
&\Big|\{x>A: \log \big(1+\frac{2}{x-1}\big)>\lambda\}\Big|\\
&=\Big|\{x\in \mathbb{R}: A<x<\frac{2}{e^{\lambda}-1}+1\}\Big|\\
&=\frac{2}{e^{\lambda}-1}+1-A.
\end{aligned}
\end{equation}
and
\begin{equation}\label{6p-3}
\begin{aligned}
&\Big|\{x<-A: -\log \big(1+\frac{2}{x-1}\big)>\lambda\}\Big|\\
&=\Big|\{x\in \mathbb{R}: \frac{2}{e^{-\lambda}-1}+1<x<-A\}\Big|\\
&=-A-\frac{2}{e^{-\lambda}-1}-1.
\end{aligned}
\end{equation}
The inequalities \eqref{6p-1}-\eqref{6p-3} imply that
\begin{equation*}
\begin{aligned}
\|[b,H](f)\chi_{E_{A}}\|_{L^{1,\infty}(\mathbb{R}^n)}&=\sup_{\lambda>0}\lambda\Big|\{x\in \mathbb{R}: |[b,H](f)(x)\chi_{E_{A}}(x)|>\lambda\}\Big|\\
&=\sup_{\lambda>0}\lambda \Big(\frac{2}{e^{\lambda}-1}-\frac{2}{e^{-\lambda}-1}-2A\Big)\\
&=\lim_{\lambda\rightarrow 0^+}\lambda \Big(\frac{2}{e^{\lambda}-1}-\frac{2}{e^{-\lambda}-1}-2A\Big)\\
&=4\neq 0,
\end{aligned}
\end{equation*}
By the fact that the functions $\frac{\lambda}{e^{\lambda}-1}$ and $-\frac{\lambda}{e^{-\lambda}-1}$ are monotonically decreasing, and
$$\lim_{\lambda\rightarrow 0^+}\frac{\lambda}{e^{\lambda}-1}=\lim_{\lambda\rightarrow 0^+}\frac{1}{e^{\lambda}}=1, \lim_{\lambda\rightarrow 0^+} \frac{\lambda}{e^{-\lambda}-1}=-1,$$
we conclude that for any $A>2$,
$$\|[b,H](f)\chi_{E_{A}}\|_{L^{1,\infty}(\mathbb{R})}=4,$$
then $[b,H](f)$ does not satisfy the condition \eqref{WLP3} in Lemma \ref{WLP}. Thus, $[b,H](f)$ is not a compact operator from $L\log L(\mathbb{R})$ to $L^{1,\infty}(\mathbb{R})$.
\end{proof}

\begin{proposition}\label{C-E-3}
There exists a function $b\in {\rm CMO}(\mathbb{R}^n)$ such that $[b,T]$ is not a compact operator from $H^{1}(\mathbb{R}^n)$ to $L^{1,\infty}(\mathbb{R}^n)$.
\end{proposition}
\begin{proof}
Without loss of generality, we only deal with $n = 1$ and $T=H.$

Let $f(x)=-x\big(\chi_{(-1,-\frac{1}{2})}(x)+\chi_{(\frac{1}{2},1)}(x)\big)$ and $b\in C^{\infty}_{c}(\mathbb{R})$ such that
\begin{equation*}
  b(x)= \left\{
   \begin{array}{ll}\vspace{1ex}
0, & |x|> 2,\\
x^{-1}, & \frac{1}{2}<|x|< 1,\\
0, & |x|<\frac{1}{4}.
   \end{array}
 \right.
\end{equation*}
Then $f\in H^{1}(\mathbb{R})\backslash  H^{1}_{b}(\mathbb{R}) $ and $b\in {\rm CMO}(\mathbb{R})$. For any $|x|>2$, we have
\begin{equation*}\label{C3}
\begin{aligned}
[b,H][b,H](f)(x)&=\int_{\mathbb{R}}\frac{b(x)-b(y)}{x-y}f(y)dy=\int_{\mathbb{R}}\frac{0-1/y}{x-y}f(y)dy\\
&=\int_{-1}^{-1/2}\frac{1}{x-y}dy+\int_{1/2}^{1}\frac{1}{x-y}dy\\
&=\log \big(1+\frac{2x}{2x^{2}-x-1}\big).
\end{aligned}
\end{equation*}
It is easy to see that $\frac{t}{2}<\log (1+t)$ for any $t>1$. Therefore,
$$\log \big(1+\frac{2x}{2x^{2}-x-1}\big)>\frac{x}{2x^{2}-x-1}>\frac{1}{x},$$
when $x>2$. It implies that for any $A>2$,
\begin{equation*}
\begin{aligned}
\|[b,H](f)\chi_{E_{A}}\|_{L^{1,\infty}(\mathbb{R}^n)}&>\|\frac{1}{(\cdot)}\chi_{E_{A^{+}}}\|_{L^{1,\infty}(\mathbb{R}^n)}\\
&=\sup_{\lambda>0}\lambda\Big|\{x\in \mathbb{R}: A<x<\frac{1}{\lambda}\}\Big|\\
&=1,
\end{aligned}
\end{equation*}
where $A^+:=\{x\in \mathbb{R}:x>A\}$. So, the condition \eqref{WLP3} in Lemma \ref{WLP} is not satisfied. We conclude that $[b,H](f)$ is not a compact operator from $H^{1}(\mathbb{R})$ to $L^{1,\infty}(\mathbb{R})$.
\end{proof}

In the case of classical Banach function spaces, the Minkowski-type inequality follows by the application
of associate space; see \cite[Lemma 4.2]{CGB}. Now, we would like to give a proof for
the Minkowski-type inequality for weak Lebesgue spaces, .

\begin{proposition}\label{M-WLP}
(Minkowski-type inequality.) Let $1<p<\infty$. Suppose that $f$ is a nonnegative measurable function on $\mathbb{R}^n\times \mathbb{R}^n$ with $\int_{\mathbb{R}^n}\|f(x,\cdot)\|_{L^{p,\infty}(\mathbb{R}^n)}dx<\infty$. Then
$$\Big\|\int_{\mathbb{R}^n}|f(x,\cdot)|dx\Big\|_{L^{p,\infty}(\mathbb{R}^n)}\leq C\int_{\mathbb{R}^n}\|f(x,\cdot)\|_{L^{p,\infty}(\mathbb{R}^n)}dx.$$
\end{proposition}
\begin{proof}
We first verify that for any bounded set $E\subset \mathbb{R}^n$ and $1<q<p<\infty$,
\begin{equation}\label{M1}
\begin{aligned}
\bigg(\int_{E}\bigg(\int_{\mathbb{R}^n}|f(x,y)|dx\bigg)^{q}dy\bigg)^{1/q}
&\leq\int_{\mathbb{R}^n}\bigg(\int_{E}|f(x,y)|^{q}dy\bigg)^{1/q}dx\\
&\lesssim |E|^{1/q-1/p}\int_{\mathbb{R}^n}\|f(x,\cdot)\|_{L^{p,\infty}(\mathbb{R}^n)}dx.
\end{aligned}
\end{equation}
Let $x\in \mathbb{R}^n$ be a fixed point. For any $\lambda>0$, we have
$$\lambda|\{y\in \mathbb{R}^n: |f(x,y)|>\lambda\}|^{1/p}\leq \|f(x,\cdot)\|_{L^{p,\infty}(\mathbb{R}^n)}.$$
Choose
$$N(x)=\|f(x,\cdot)\|_{L^{p,\infty}(\mathbb{R}^n)}\big(\frac{q}{p-q}\big)^{1/p}|E|^{-1/p},$$
For $1<q<p<\infty$, we conclude that
%$$E_{\lambda}=\big\{y\in E: |f(x,y)|>\lambda\big\}.$$
%Then
\begin{align*}
\int_{E}|f(x,y)|^{q}dy&=q\int_{0}^{\infty}\lambda^{q-1}\Big|\big\{y\in E: |f(x,y)|>\lambda\big\}\Big|d\lambda\\
&\leq q\int_{0}^{N(x)}\lambda^{q-1}|E|d\lambda +q\int_{N(x)}^{\infty}\lambda^{q-1}\frac{\|f(x,\cdot)\|^{p}_{L^{p,\infty}(\mathbb{R}^n)}}{\lambda^{p}}d\lambda\\
&=N(x)^{q}|E|+ \frac{q}{p-q}N(x)^{q-p}\|f(x,\cdot)\|^{p}_{L^{p,\infty}(\mathbb{R}^n)}.
\end{align*}
It shows that
\begin{align*}
\bigg(\int_{E}|f(x,y)|^{q}dy\bigg)^{1/q}&\leq 2\Big(\frac{q}{p-q}\Big)^{1/p}\|f(x,\cdot)\|^{p}_{L^{p,\infty}(\mathbb{R}^n)}|E|^{-1/p+1/q}.
\end{align*}
Therefore, the proof of the inequality \eqref{M1} is completed.

Now we return to our proof. For any $\lambda>0$, take $E=\{y\in \mathbb{R}^n: \int_{\mathbb{R}^n}|f(x,y)|dx>\lambda\}$, then $|E|<\infty.$ In fact, if $|E|=\infty$, there is a sequence $\{E_{k}\}$ of measurable sets such that $E_{k}\subset E$ and $|E_{k}|=k$ for $k=0,1,2,\cdots.$ Thus for every $k$, by \eqref{M1}, we get
\begin{align*}
\lambda^{q}k=\lambda^{q}|E_{k}|&\leq \int_{E_{k}}\bigg(\int_{\mathbb{R}^n}|f(x,y)|dx\bigg)^{q}dy\\
&\leq C |E_{k}|^{1-q/p}\bigg(\int_{\mathbb{R}^n}\|f(x,\cdot)\|_{L^{p,\infty}(\mathbb{R}^n)}dx\bigg)^{q}\\
&\leq Ck^{1-q/p}\bigg(\int_{\mathbb{R}^n}\|f(x,\cdot)\|_{L^{p,\infty}(\mathbb{R}^n)}dx\bigg)^{q}.
\end{align*}
However, it is not true.

It follows from the inequality \eqref{M1} that
\begin{align*}
\lambda^{q}|E|&\leq \int_{E}\bigg(\int_{\mathbb{R}^n}|f(x,y)|dx\bigg)^{q}dy\\
&\leq C |E|^{1-q/p}\bigg(\int_{\mathbb{R}^n}\|f(x,\cdot)\|_{L^{p,\infty}(\mathbb{R}^n)}dx\bigg)^{q}.
\end{align*}
Therefore,
$$\lambda|E|^{1/p}\leq C\int_{\mathbb{R}^n}\|f(x,\cdot)\|_{L^{p,\infty}(\mathbb{R}^n)}dx,$$
then
$$\Big\|\int_{\mathbb{R}^n}|f(x,\cdot)|dx\Big\|_{L^{p,\infty}(\mathbb{R}^n)}\leq C\int_{\mathbb{R}^n}\|f(x,\cdot)\|_{L^{p,\infty}(\mathbb{R}^n)}dx.$$
Therefore, this proof has been completed.
\end{proof}
\subsection*{Acknowledgements}
The authors extend heartfelt gratitude to Professor C. P\'{e}rez for introducing the problem pertaining to the endpoint theory for the boundedness of commutators.

\section*{Conflict of interests}
The authors declare that they have no conflict of interest.

\subsection*{Data availability}
Data sharing not applicable to this article as no data sets were generated or analysed during
the current study.

\subsection*{Declarations}
{\bf Conflict of interest} Conflict of interest no potential conflict of interests was reported by the authors.


\begin{thebibliography}{999}

\bibitem{B.O}
O. Blasco, A. Ruiz and L. Vega, Non-interpolation in Morrey-Campanato and Block Spaces.  Ann. Sc.
Norm. Super. Pisa Cl. Sci.  28(5)(1999), 31-40.

\bibitem{CGB}
A. Caetano, A. Gogatishvili and B. Opic, Compactness in quasi-Banach
function spaces and applications to compact embeddings of Besov-type spaces. Proc. Roy. Soc.
Edinburgh Sect. A 146 (2016), no. 5, 905--927.

\bibitem{CT2015}
L. Chaffee and R.H. Torres, Characterization of compactness of the commutators of bilinear fractional
integral operators. Potential Anal., 43(3)(2015), 481-494.

\bibitem{CH2015}
J.C. Chen and G.E. Hu, Compact commutators of rough singular integral operators. Canad. Math. Bull.,
58(1)(2015), 19-29.

\bibitem{CDW2009}
Y.P. Chen, Y. Ding and X.X. Wang. Compactness of commutators of Riesz potential on Morrey spaces.
Potential Anal., 30(4)(2009), 301-313.

\bibitem{CDW}
Y.P. Chen, Y. Ding and X.X. Wang, Compactness of commutators for singular integrals on Morrey spaces.
Can. J. Math., 64(2012), 257--281.

\bibitem{CH2001}
W.G. Chen and G.E. Hu, Weak type $(H^1(\mathbb{R}^n), L^1(\mathbb{R}^n))$ estimate for a multilinear singular integral operator. Adv. in
Math. (China), 2001, 30(1): 63.

\bibitem{CLMS1993}
R. Coifman, P.-L. Lions, Y. Meyer and S. Semmes, Compensated compactness and
Hardy spaces. J. Math. Pures Appl., (9)72 (1993), 247--286.

\bibitem{C1974}
R.R. Coifman, A real variable characterization of $H^p$. Studia Math., 51(1974), 269-274.

\bibitem{CRW1976}
R.R. Coifman, R. Rochberg and G. Weiss,
Factorization theorems for Hardy spaces in several variables.
Ann. of Math., 103(1976), 611--635.

\bibitem{DW2021}
N.A. Dao and B.D. Wick, Hardy factorization in terms of Multilinear Calder\'{o}n-Zygmund operators using Morrey spaces. Potential Anal., 59(2023), 41--64.

\bibitem{DS1984}
R.A. Devore and R.C. Sharpley, Maximal functions measuring smoothness. Mem. Amer. Math. Soc., 47(1984).

\bibitem{P.B}
P. Duren, B. Romberg and A. Shields, Linear functionals on $H^p$ spaces with $0 < p < 1$. J. Reine Angew. Math., 238(1969), 32--60.

\bibitem{FS1972}
C. Fefferman and E.M. Stein, $H^{p}$ spaces of several variables. Acta Math., 129(1972), 137--193.


\bibitem{GF1985}
J. Garcia-Cuerva and J. L. Rudio de Francia, Weighted norm inequalities and related topics, North-Holland Mathematics Studies {\bf 116}, North-Holland Publishing Co. Amsterdam, 1985.

\bibitem{G2008}
L. Grafakos, Classical Fourier analysis. Second ed., Graduate Texts in Mathematics, 249, Springer,
New York, 2008.

\bibitem{GT2002}
L. Grafakos and R.H. Torres, Maximal operator and weighted norm inequalities for multilinear singular integrals.
Indiana U. Math. J., 51(2002), 1261--1276.


\bibitem{GHWY}
W.C. Guo, J.X. Hu, H.X. Wu and D.Y. Yang, Boundedness and compactness of commutators associated with Lipschitz functions.
Anal. Appl., 20(1)(2022), 35-71.

\bibitem{GWY1}
W.C. Guo, H.X. Wu and D.Y. Yang, A revisit on the compactness of commutators.
Can. J. Math.,  73(6)(2021), 1667--1697.


\bibitem{HT2019}
J. Hart and R.H. Torres, John-Nirenberg inequalities and weight invariant BMO
spaces. J. of Geom. Anal., 29(2019), 1608--1648.

\bibitem{J1978}
S. Janson, Mean oscillation and commutators of singular integral operators. Arkiv for
Matematik, 16(1978), 263--270.

\bibitem{JN1961}
F. John and L. Nirenberg, On functions of bounded mean oscillation. Comm. Pure Appl. Math., 2(1961), 415--426.

\bibitem{Str1979}
J.-O. Str\"{o}mberg, Bounded mean oscillation with Orlicz norms and duality of Hardy spaces. Indiana
Univ. Math. J., 28(3)(1979), 511--544.

\bibitem{KL12001}
S. Krantz and S.-Y. Li, Boundedness and compactness of integral operators on spaces of homogeneous
type and applications, I. J. Math. Anal. Appl., 258(2001), 629--641.

\bibitem{KL22001}
S. Krantz and S.-Y. Li, Boundedness and compactness of integral operators on spaces of homogeneous
type and applications, II. J. Math. Anal. Appl., 258(2001), 642--657.

\bibitem{Kol}
A. Kolmogoroff, Ueber kompaktheit dr funktionenmengen bei der konvergenz im mittel, Nachrichten von der Gesellschaft der Wissenschaften zu Gottingen. MathematischPhysikalische Klasse, (1931) 60-63.

\bibitem{LPPW2012}
M.T. Lacey, S. Petermichl, J.C. Pipher and B.D. Wick, Multi-parameter Div-Curl
lemmas. Bull. Lond. Math. Soc., 44 (2012), 1123-1131.

\bibitem{L1978}
R.H. Latter, A characterization of $H^{p}(\mathbb{R}^n)$ in terms of atoms. Studia
Math., 62 (1978), 93--101.

\bibitem{LSY}
S.-Y. Li, Characterization of the boundedness for a family of commutators on $L^{p}$. Colloquium Mathematicum, 70(1996), 59--71.

\bibitem{LWY}
S.Z. Lu, Q. Wu and D.C. Yang, Boundedness of commutators on Hardy type spaces. Sci. China Ser. A-Math., 45(2002), 984--997.

\bibitem{NPTV2002}
F. Nazarov, G. Pisier, S. Treil and A. Volberg, Sharp estimates in vector Carleson
imbedding theorem and for vector paraproducts. J. Reine Angew. Math., 542(2002),
147--171.

\bibitem{N1957}
Z. Nehari, On bounded bilinear forms. Ann. of Math., (2)65(1957), 153--162.

\bibitem{NS2017}
T. Nogayama and Y. Sawano, Compactness of the commutators generated by Lipschitz functions and
fractional integral operators. Mat. Zametki, 102(5)(2017), 749-760.

\bibitem{P1995}
C. P\'{e}rez, Endpoint estimates for commutators of singular integral operators. J. Func. Anal., 128(1975), 163-185.


\bibitem{Rie}
M. Riesz, Sur les ensembles compacts de fonctions sommables. Acta. Szeged Sect. Math.,
(1933), 136-142.

\bibitem{Tik}
V. Tikhomirov, On the compactness of sets of functions in the case of convergence in mean.
Selected Works of A. N. Kolmogorov, Springer (1991), 147-150.

\bibitem{W-SM}
D.H. Wang, The necessity theory for commutators of multilinear singular integral operators: the
weighted case. Studia. Math., Online first.

\bibitem{WZTmn}
D.H. Wang, J. Zhou and Z.D. Teng, A note on Campanato spaces and its application. Math. Notes, 103(2018), 483--489.

\bibitem{WZTams}
D.H. Wang, J. Zhou and Z.D. Teng, Characterizations of weighted BMO space and its application. Acta Math. Sin., 37(2021), 1278--1292.

\bibitem{AMP}
D.H. Wang, J. Zhou and Z.D. Teng, Characterization of CMO via compactness of the commutators of bilinear fractional integral operators. Anal. Math. Phys., 9(2019), 1669--1688.

\bibitem{WZ-JMAA}
D.H. Wang and R.X. Zhu, Weak factorizations of the Hardy space in terms of multilinear fractional
integral operator. J. Math. Anal. Appl., 517(1)(2023), 12pp.

\bibitem{WZS-AFA}
D.H. Wang, R.X. Zhu and L.S. Shu, The factorizations of $H^\rho(\mathbb{R}^n)$ via multilinear Calder\'{o}n-Zygmund operators on weighted Lebesgue spaces. Ann. Funct. Anal., 14(2023), 47.

\bibitem{U1}
A. Uchiyama, On the compactness of operators of Hankel type. Tohoku. Math. J., 30(1978), 163--171.

\bibitem{U2}
A. Uchiyama, A constructive proof of the Fefferman-Stein decomposition of ${\rm BMO}(\mathbb{R}^n)$.
Acta Math., 148 (1982), 215--241.
\end{thebibliography}
\end{document}